\documentclass[11pt,a4paper]{article}

\usepackage[utf8]{inputenc}
\usepackage[T1]{fontenc}
\usepackage[english]{babel}
\usepackage{xcolor}
\usepackage{graphicx}
\usepackage{subcaption}
\usepackage{hyperref}
\usepackage{amsmath, amstext, amsopn, amssymb, amsthm}
\usepackage{keytheorems}
\usepackage{mathtools}
\usepackage{siunitx}
\usepackage{etoolbox}
\usepackage{bm}
\usepackage{makebox}
\usepackage{stmaryrd}
\usepackage[left=2.5cm,right=2.5cm,top=3cm,bottom=3cm]{geometry}
\usepackage[shortlabels]{enumitem}
\usepackage{natbib}

\addto\extrasenglish{

}

\hypersetup{
  colorlinks   = true,
  urlcolor     = blue,
  linkcolor    = blue,
  citecolor    = blue
}

\graphicspath{{figures/}}

\DeclareMathOperator{\supp}{supp}

\DeclarePairedDelimiter{\norm}{\lVert}{\rVert}
\DeclarePairedDelimiter{\abs}{\lvert}{\rvert}
\DeclarePairedDelimiterX{\inner}[2]{\langle}{\rangle}{{#1},{#2}}
\DeclarePairedDelimiter{\bbrac}{\llbracket}{\rrbracket}

\newcommand{\maxsub}[2]{\max_{\substack{#1 \\ #2}}}

\DeclarePairedDelimiterXPP{\EV}[1]{\mathop{}\!\mathbb{E}}{[}{]}{}{#1}
\DeclarePairedDelimiterXPP{\EVV}[2]{\mathop{}\!\mathbb{E}_{#1}}{[}{]}{}{#2}
\DeclarePairedDelimiterXPP{\Prob}[1]{\mathop{}\!\mathbb{P}}{(}{)}{}{#1}
\renewcommand{\Pr}[1]{\Prob{#1}}

\newcommand{\I}{\mathbb{I}}
\newcommand{\R}{\mathbb{R}}
\newcommand{\N}{\mathbb{N}}
\newcommand{\K}{\mathcal{K}}
\newcommand{\calX}{\mathcal{X}}

\newcommand{\calN}{\mathcal{N}}
\newcommand{\eps}{\varepsilon}

\newcommand{\bmu}{\bar{\mu}}

\newcommand{\hmu}{\hat{\mu}}
\newcommand{\hnu}{\hat{\nu}}
\newcommand{\hG}{\hat{G}}

\newcommand{\rmf}{L}
\newcommand{\rmb}{B}
\newcommand{\Tf}[1][]{T_n^{\rmf\IfBlankTF{#1}{}{,#1}}}

\newcommand{\Tb}{T_n^{\rmb}}

\newcommand{\nullf}[1][]{H_0^{\rmf\IfBlankTF{#1}{}{,#1}}}
\newcommand{\nullb}{H_0^{\mathrm{I}}}

\newcommand{\hphif}{\hat{\phi}^{\rmf}}

\newcommand{\hq}{\hat{q}}

\newcommand{\defeq}{\coloneqq}
\newcommand{\eqdef}{\eqqcolon}
\makeatletter
\newcommand*{\transp}{%
	{\mathpalette\@transpose{}}%
}
\newcommand*{\@transpose}[2]{%
	\raisebox{\depth}{$\m@th#1\intercal$}%
}
\makeatother

\newcommand{\idn}[1]{\bbrac{#1}}
\makeatletter
\newcommand*{\bigcdot}{}%
\DeclareRobustCommand*{\bigcdot}{%
	\mathbin{\mathpalette\bigcdot@{}}%
}
\newcommand*{\bigcdot@scalefactor}{.5}
\newcommand*{\bigcdot@widthfactor}{1.15}
\newcommand*{\bigcdot@}[2]{%
	\sbox0{$#1\vcenter{}$}%
	\sbox2{$#1\cdot\m@th$}%
	\hbox to \bigcdot@widthfactor\wd2{%
		\hfil
		\raise\ht0\hbox{%
			\scalebox{\bigcdot@scalefactor}{%
				\lower\ht0\hbox{$#1\bullet\m@th$}%
			}%
		}%
		\hfil
	}%
}
\makeatother
\newcommand{\argdot}{\,\bigcdot\,}

\newcommand{\OText}[0]{\OT^{\pm}_{c}}

\newcommand{\restr}[2]{{
\left.\kern-\nulldelimiterspace 
#1 
\vphantom{\big|} 
\right|_{#2} 
}}

\newcommand{\dconv}{\xrightarrow{\mathmakebox[0.5cm]{\mathcal{D}}}}

\newcommand{\asconv}{\xrightarrow{\mathmakebox[0.5cm]{\mathrm{a.s.}}}}
\newcommand{\asntoinfty}{\text{as $n_\star \to \infty$}}

\DeclareMathOperator{\Var}{Var}
\DeclareMathOperator{\Cov}{Cov}

\DeclareMathOperator{\OT}{OT}
\DeclareMathOperator{\colsum}{col\Sigma}
\DeclareMathOperator{\rowsum}{row\mspace{-2mu}\Sigma}

\DeclareMathOperator{\Mult}{Mult}
\newcommand{\normaldist}{\mathop{}\!\mathcal{N}}

\DeclareMathAlphabet{\mymathbb}{U}{BOONDOX-ds}{m}{n}
\newcommand{\bbzero}{\mymathbb{0}}

\theoremstyle{plain}
\newkeytheorem{theorem}[parent=section, refname=Theorem]
\newkeytheorem{proposition}[sibling=theorem, refname=Proposition]
\newkeytheorem{lemma}[sibling=theorem, refname=Lemma]
\newkeytheorem{corollary}[sibling=theorem, refname=Corollary]
\theoremstyle{definition}
\newkeytheorem{example}[sibling=theorem, refname=Example]
\newkeytheorem{definition}[sibling=theorem, refname=Definition]
\newkeytheorem{remark}[sibling=theorem, refname=Remark]

\newcommand{\footremember}[2]{
	\footnote{#2}
	\newcounter{#1}
	\setcounter{#1}{\value{footnote}}
}
\newcommand{\footrecall}[1]{
	\footnotemark[\value{#1}]
}

\begin{document}

\title{\bf Optimal Transport Based Testing \\ in Factorial Designs}
\author{\begin{tabular}{ c c c }
	Michel Groppe\hspace{-0.3em}\footremember{ims}{Institute for Mathematical Stochastics, University of G{\"o}ttingen, Goldschmidtstra{\ss}e 7, 37077 G{\"o}ttingen, Germany} &
  Linus Niemöller\hspace{-0.3em}\footrecall{ims} &
  Shayan Hundrieser\hspace{-0.1em}\footrecall{ims}\!\!\footremember{tw}{Department of Statistics and Data Science, Yale University, 219 Prospect St, New Haven, CT, USA}
  \end{tabular} \\[1em]
  \begin{tabular}{c c c c}
    David Ventzke\hspace{-0.3em}\footremember{xphy}{Institute for X-Ray Physics, University of G{\"o}ttingen, Friedrich-Hund-Platz 1, 37077 G{\"o}ttingen, Germany} &
  Anna Blob\hspace{-0.3em}\footrecall{xphy} &
  Sarah Köster\hspace{-0.3em}\footrecall{xphy}\hspace{-0.3em}\footremember{mbexc}{Cluster of Excellence ``Multiscale Bioimaging: from Molecular Machines to Networks of Excitable Cells'' (MBExC), University of G{\"o}ttingen, Robert-Koch-Stra{\ss}e 40, 37075 G{\"o}ttingen, Germany}
  &  Axel Munk\hspace{-0.3em}\footrecall{ims}\hspace{-0.3em}\footrecall{mbexc}
\end{tabular}}
\maketitle

\vspace{-0.75cm}
\begin{center}
  Email address: \href{mailto:munk@math.uni-goettingen.de}{munk@math.uni-goettingen.de}
\end{center}
\vspace{0.5cm}

\begin{abstract}
We introduce a general framework for testing statistical hypotheses in factorial designs for probability measures supported on discrete spaces. The suggested methodology is based on the pairwise comparison of measures using optimal transport (OT). The formulation of hypotheses is intuitive: It is a direct extension of those underlying the analysis of variance (ANOVA) and its nonparametric counterparts to test for linear relationships between (discrete) probability measures in factorial designs. To this end, means or cumulative distribution functions simply will be replaced by measures. We derive under the null hypotheses and under (local) alternatives the asymptotic distribution of the corresponding empirical OT test statistic, which is the optimal value of a linear program with random objective function. It turns out that this requires to extend existing techniques from probability measures to signed measures, and we show directional Hadamard differentiability and the validity of the functional delta method. We discuss computational issues, permutation and bootstrap tests, and back up our findings with simulations. We illustrate our methodology on datasets from cellular biophysics and from biometric fingerprint identification.
\end{abstract}

\noindent%

\vfill

\newpage

\section{Introduction} \label{sec:intro}

Comparing the distributions of multiple datasets is an everlasting task in many disciplines. Based on independent samples from $\K \geq 2$ probability distributions $\mu^k$, $k \in \idn{\K} \defeq \{1, \ldots, \K\}$, this can be done by testing statistical hypotheses whether a certain relationship between the underlying $\mu^1, \ldots, \mu^{\K}$ holds, e.g., in a one-way layout the null hypothesis $H_0 : \mu^1 = \ldots = \mu^{\K}$. Tests for $H_0$ and more complex hypotheses (see \citealp{Gelman2005} for its relevance for modern data analysis) have a long history, e.g., the celebrated $F$-tests \citep{fisher1925statistical} within the context of analysis of variance for normally distributed data (see e.g.\ \citealp{Freedman2009}) and its nonparametric analogs (see e.g. \citealp{Hettmansperger1984}), such as the celebrated Kruskal-Wallis test \citep{Kruskal1952}. Other methods include tests for categorial data (see e.g. \citealp{Agresti2012}) such as for the analysis of multinomial probabilities. In general, all these tests employ specific properties of the underlying data generating process, e.g., that samples come from a uni- or multivariate (normal) distribution or correspond to empirical cell frequencies as in categorial data analysis.

To elaborate further, consider the framework of nonparametric analysis of variance (ANOVA) in factorial designs which extends parametric ANOVA for means in a natural way to cumulative distribution functions $F$ (c.d.f.s.) of the underlying probability measures $\mu$ (see e.g. the survey in \citealp{Akritas2025}). For instance, the one-way layout reads $H_0 : F^1 = \ldots = F^{\K}$, with $F^k$ being the c.d.f.\ of a real-valued random variable, i.e., $\mu^k$, $k \in \idn{\K}$, are Borel probability measures on the reals. In the two-way layout, see e.g.\ \citet[Section~2.2]{Brunner1996}, we observe samples from $\K = \K_1 \K_2$ c.d.f.s.\ $F^{ij}$, $i \in \idn{\K_1}$, $j \in \idn{\K_2}$, that depend on two factors A and B of size $\K_1$ and $\K_2$, respectively. Then, in addition to hypotheses on the main effects, the null hypothesis for an interaction effect between A and B can be stated as $H_0 : F^{ij} - \bar{F}^{ij} = 0$, $i \in \idn{\K_1}$, $j \in \idn{\K_2}$, where $\bar{F}^{ij} \defeq \frac{1}{\K_1} \sum_{r=1}^{\K_1} F^{rj} + \frac{1}{\K_2} \sum_{s=1}^{\K_2} F^{is} - \frac{1}{\K_1 \K_2} \sum_{r,s=1}^{\K_1,\K_2} F^{rs}$. Note that this formulation is analog to the null hypothesis in the classical ANOVA framework for the two-way layout by substituting expectations by c.d.f.s. More general, these hypotheses and those for higher-way layouts can be formulated and tested via (see e.g. \citealp{Akritas1997,Brunner1997,Konietschke2023}) $H_0 : LF = \bbzero$ where $L \in \R^{M \times \K}$ is a suitable matrix, $F$ the vector of c.d.f.s.\ and $\bbzero$ the zero vector.

While testing such hypotheses in factorial designs is well-established for Euclidean and count data, say, a unifying theory for more complex data, such as graphs, networks or more generally, metric data remains elusive to some extent. Notable exceptions include graph- or kernel-based tests for the one-way layout and non-Euclidean data \citep{Song2022,Huang2024,Zhang2022}, and distance-based tests for the one-way \citep{Dubey2019} or higher-way layouts \citep{Anderson2001}, see \citet{Mueller2024} for a recent account. Closest to our work is \citet{Kravtsova2025} who take a first step into the direction of optimal transport (OT) based tests, restricted to the one-way layout. However, a treatment of general factorial designs is missing. In this work, we aim to fill this gap and provide methodology for testing hypothesis for measures $\mu^{k}$, $k \in \idn{\K}$, with countable support in general factorial designs. We will see that this requires to extend statistical OT from probability measures to general signed measures, a topic which is of interest on its own.

Besides its conceptual similarity to classical ANOVA, a particular appeal of the present approach is its intuitive geometric interpretation due to the flexible choice of the underlying cost. This allows to equip the space of probability measures with a natural metric structure that is aligned with the geometry of the underlying ground space metric. Various applications document the resulting Wasserstein distance to be robust under spatial deformations while capturing relevant features, such as the most efficient morphing between objects \citep{Schiebinger2019,Tameling2021,Bonneel2023,Yuan2025}, and often comes close to human perception of similarity \citep{Rubner2000}. We see this as a particular benefit in our own application on fingerprint identification, see \autoref{sec:fingerprints}.

Inspired by (non-)parametric testing in factorial designs, we propose in a first step to substitute the means resp.\ c.d.f.s.\ by the corresponding probability measures themselves and then to imitate the corresponding hypotheses. For example, the null hypothesis for no interaction effect between the factors A and B in the two-way layout becomes
\begin{equation} \label{eq:interactionAB}
  H_0 : \mu^{ij} - \mu^{i\bullet} - \mu^{\bullet j} + \mu^{\bullet\bullet} = 0\,, \qquad \text{for all } i \in \idn{\K_1},\,j \in \idn{\K_2}\,,
\end{equation}
where 
\begin{equation} \label{eq:mu_mean}
    \mu^{i\bullet} \defeq \frac{1}{\K_2}\sum_{s=1}^{\K_2} \mu^{is}\,, \quad \mu^{\bullet j} \defeq \frac{1}{\K_1}\sum_{r=1}^{\K_1}\mu^{rj}\,, \quad \mu^{\bullet\bullet} \defeq \frac{1}{\K_1\K_2}\sum_{r,s=1}^{\K_1,\K_2}\mu^{rs}\,.
\end{equation} 
More generally, hypotheses in higher-way layouts read as
\begin{equation} \label{eq:null_fac}
  \nullf : L \mu = \bbzero\,,
\end{equation} 
where $\mu \defeq [\mu^1, \ldots, \mu^{\K}]^\transp$ and the row sums of $L \in \R^{M \times \K}$ are all equal to $0$. We stress that the classical formulations in terms of c.d.f.s.\ implicitly assume the ground space to be a subset of the reals or the $d$-dimensional Euclidean space, while our new formulation has no such restriction and applies to general discrete measures on a metric space. In particular, it does not require any kind of total ordering on the ground space.  

To motivate our test statistics for the generalized hypotheses \eqref{eq:null_fac}, we take inspiration from the classical one-way ANOVA statistic for samples $X^1, \ldots, X^{\K}$ in $\R^d$ given by
\begin{equation} \label{eq:idea}
  \min_{Y \in \R^d} \frac{1}{\K}\sum_{k=1}^{\K} \norm{X^k - Y}^2=\frac{1}{\K^2} \sum_{i,j=1}^{\K} \norm{X^i-X^j}^2,
\end{equation}
where $\norm{\argdot}$ denotes the Euclidean norm and the minimum is achieved at the empirical mean $\bar{Y} = \frac{1}{\K}\sum_{k=1}^{\K} X^k$. Note, that the left-hand side (l.h.s.) of \eqref{eq:idea} is equal to the Fr\'echet variance of the barycenter of $X^1, \ldots, X^{\K}$ with respect to $\norm{\argdot}^2$, while the right-hand side (r.h.s.) resembles all pairwise distances. Similar in spirit to \citet{Anderson2006} and \citet{Mueller2024}, we propose to imitate the r.h.s.\ of \eqref{eq:idea} relating the pairwise empirical OT distances between measures. Leveraging novel distributional limits for OT between signed measures, we then can design statistical tests for all linear relationships \eqref{eq:null_fac} such as for the main effects and interaction effect in a two-way layout.

We stress that equality of the l.h.s.\ and r.h.s.\ of \eqref{eq:idea} does not hold for the OT based counterparts, in general, due to the fact that the Wasserstein space is not flat, in general \citep{Panaretos2020}. Thus, replacing the l.h.s.\ in \eqref{eq:idea} by its OT counterpart does lead to a different statistic. However, a simple sandwich inequality asserts they are not further apart than by a factor of $2$ (see \autoref{thm:ineq_apad_bary}). Although the values corresponding to the l.h.s.\ and r.h.s.\ of \eqref{eq:idea} are close, below in \autoref{subsec:approach}, we will argue that the approach corresponding to the r.h.s.\ based on pairwise distances is in general superior by statistical and computational reasons, and therefore advocated in this paper.

\subsection{Optimal Transport for Signed Measures}

To be more specific, consider now the extension of \eqref{eq:null_fac} to our context. To this end, we need to introduce OT between signed measures. To ease exposition, we restrict in the following to finite ground spaces, however, we stress that our methodology extends (with some technical effort) to countable spaces and even beyond (see \autoref{rem:extensions}). For $N \in \N$ and $E > 0$, denote the (scaled) simplex as $\Delta_N^{(E)} \defeq \{\mu\in\R^N : \sum_{i=1}^{N}\mu_i=E,\mu\geq 0 \}$, which for $E = 1$ gives the probability simplex $\Delta_N \defeq \Delta_N^{(1)}$. Let $c \in \R^{N \times N}$ be a \textit{metric cost matrix}, i.e., it holds that $c \geq 0$ with $c_{ij} = 0$ if and only if $i = j$, $c_{ij} = c_{ji}$ and $c_{ij} \leq c_{ir} + c_{rj}$ for all $i, j, r \in \idn{N}$. Then, the OT distance between $\mu$ and $\nu$ (and its empirical counterparts of the r.h.s.\ of \eqref{eq:idea}) in the probability simplex $\Delta_N$ (or more generally $\Delta_N^{(E)}$) is defined as
\begin{equation*}
  \OT_c(\mu, \nu) \defeq \min_{\pi \in \Pi(\mu, \nu)} \inner{c}{\pi}\,,
\end{equation*}
where $\Pi(\mu, \nu)$ is the set of transport plans between the probability vectors $\mu$ and $\nu$, i.e., $N \times N$ non-negative matrices with row and column sums equal to $\mu$ and $\nu$, respectively (see e.g.\ \citealp{Santambrogio2015,Peyre2019}). However, working with probability distributions is not sufficient for our purpose. Notably, this does not become apparent in a one-way layout. However, already in the two-way layout as in \eqref{eq:interactionAB} the interaction term $\mu^{ij} - \mu^{i\bullet} - \mu^{\bullet j} + \mu^{\bullet\bullet}$ may contain negative entries, i.e., it is a signed measure in $\Delta_N^\pm \defeq \{\mu\in\R^N : \sum_{i=1}^N\mu_i=0\}$. This requires to extend the OT functional to signed measures. For $\mu \in \Delta_N^\pm$ write the Jordan decomposition as
\begin{equation}
\mu=\mu^+-\mu^-\,,\quad\mu^+ \defeq \max(\mu,0)\,, \quad \mu^- \defeq -\min(\mu,0)\,.
\end{equation}
Then, the extended OT functional for $\mu,\, \nu \in \Delta_N^\pm$ is given by
\begin{equation} \label{eq:ot_ext}
\OT_c^\pm(\mu,\nu) \defeq \OT_c(\mu^+ + \nu^-, \nu^+ + \mu^-)\,.
\end{equation}
Note that for $\mu,\nu\in\Delta_N^\pm$ it holds $\sum_{i=1}^N(\mu^+_i-\mu^-_i)=\sum_{i=1}^N(\nu^+_i-\nu^-_i)$, so $\mu^++\nu^-$ and $\nu^++\mu^-$ are both measures with non-negative entries and the same total mass. Intuitively, the extended OT notion allows transport between the positive and negative parts of $\mu$ and $\nu$ separately, as well as cancellation of mass between the positive part and negative part of $\mu$ and $\nu$ themselves, see \autoref{fig:signed_ot} for illustration. This extension has been introduced in \citet{Ambrosio2011,Mainini2012} in the context of certain nonlinear partial differential equations. To the best of our knowledge, here we introduce it for statistical purposes for the first time. As shown in \citet{Mainini2012}, as $c$ is a metric cost matrix, $\OText$ is a metric on $\Delta_N^\pm$, as well. In particular, it follows that $\OText(\mu, \nu) = 0$ if and only if $\mu = \nu$. This enables us to use the extended OT cost to test general hypothesis as given in \eqref{eq:null_fac} (see \autoref{subsec:approach}). Furthermore, it allows us to obtain explicit expressions for the asymptotic distribution and power of the corresponding test statistics in terms of $\OText$ applied to Gaussian limiting processes (see \autoref{subsec:main_res}).

\begin{figure}
  \centering
  \includegraphics[scale=1.2]{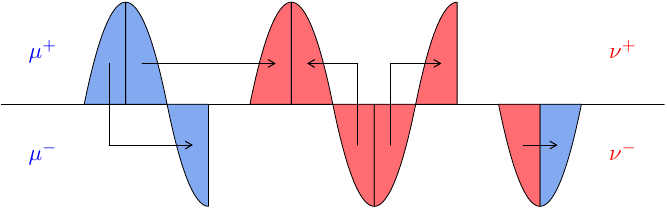}
  \caption{Illustration of the OT between the signed measures $\color{blue} \mu = \mu^+ - \mu^-$ and $\color{red} \nu = \nu^+ - \nu^-$. An arrow between two parts indicates transport. Note that we can observe two types of transport here: Transport between $\mu$ and $\nu$ of the same sign ($\mu^+ \to \nu^+$ and $\nu^- \to \mu^-$), as well as cancellation ($\mu^+ \to \mu^-$ and $\nu^- \to \nu^+$).
  } \label{fig:signed_ot}
\end{figure}

Finally, we introduce the barycenter OT analog to the l.h.s.\ of \eqref{eq:idea}. For $\mu^1, \ldots, \mu^{\K} \in \Delta_N$ and positive weights $w \in \Delta_{\K}$, the OT barycenter functional is defined as
\begin{equation} \label{eq:ot_bary}
  B_c^w(\mu^1, \ldots, \mu^{\K}) \defeq \min_{\nu \in \Delta_N} \sum_{k=1}^{\K} w_k \OT_c(\mu^k, \nu)\,.
\end{equation}
Similarly to the OT costs, as $c$ is a metric, it follows that $B_c^w(\mu^1, \ldots, \mu^{\K}) = 0$ if and only if $\mu^1 = \ldots = \mu^{\K}$.

\subsection{The Proposed Approach} \label{subsec:approach}

Suppose that $\mu^1,\ldots,\mu^{\K} \in \Delta_N$, $\K \geq 2$, are (unobservable) probability vectors generating (observable) independent samples $X_1^k, \ldots, X^k_{n_k} \sim \mu^k$, $k \in \idn{\K}$, with sample sizes $n = (n_1,...,n_{\K}) \in \N^{\K}$. Write $\mu \defeq [\mu^1, \ldots, \mu^{\K}]^\transp$. In the following, we introduce the factorial design OT test (FDOTT) for null hypotheses $\nullf$ as given in \eqref{eq:null_fac}.

Denote $\hmu^k_{n_k}$ the empirical probability vectors (relative cell frequencies) based on the samples $X_1^k, \ldots, X_{n_k}^k$, i.e.,
\begin{equation} \label{eq:empirical_prob_vec}
  \hmu^k_{n_k,i} = \frac{1}{n_k} \sum_{r=1}^{n_k} \I(X_r^k = x_i)  \qquad \text{for all } i \in \idn{N},\,k \in \idn{\K}\,,
\end{equation}
where $\calX \defeq \{x_1, \ldots, x_N \}$ is the underlying finite ground space of size $N$. Write $\hmu_n \defeq [\hmu^1_{n_1}, \ldots, \hmu^{\K}_{n_{\K}}]^\transp$ and define for the given sample sizes $n = (n_1, \ldots, n_{\K})$ the coefficient
\begin{equation}\label{eq:rho}
\rho_n \defeq \left( \prod_{k=1}^{\K} n_k\right)/\left(\sum_{k=1}^{\K}\prod_{j=1,j\neq k}^{\K} n_j\right)\,.
\end{equation}
To detect deviations from the null hypothesis $\nullf$ in \eqref{eq:null_fac}, we employ the FDOTT statistic
\begin{equation}\label{eq:test_stat_fac}
  \Tf(\hmu_n) \equiv T^L_{n,s,c}(\hmu_n) \defeq \frac{\sqrt{\rho_n}}{s}\sum_{m=1}^M\OText([L\hmu_n]_m, \bbzero),
\end{equation}
where $s > 0$ depends on $\K$ and $M$ only and is used for scaling purposes. Here, the vector $[L\hmu_n]_m$ denotes the $m$-th row of $L\hmu_n$.

\begin{remark}[Comparison to barycenter based testing] \label{rem:barycenter_method}
  As advocated in \citet{Kravtsova2025}, in the case of the one-way layout, i.e., with null hypothesis
  \begin{equation} \label{eq:null_bary}
    \nullb : \mu^1 = \ldots = \mu^{\K}\,,
  \end{equation}
  we can also use the OT barycenter \eqref{eq:ot_bary} to design a test statistic,
  \begin{equation}\label{eq:test_stat_bary}
    \Tb(\hmu_n) \equiv T^B_{n,w,c}(\hmu_n) \defeq \sqrt{\rho_n}\,B_c^w(\hmu^1_{n_1}, \ldots, \hmu^{\K}_{n_{\K}})\,.
  \end{equation}
  Then, $\Tb(\hmu_n)$ can be used as an alternative to $\Tf(\hmu_n)$ in the one-way layout. However, as we will see, in general, the barycenter method has three major drawbacks compared to FDOTT. First, the test statistic $\Tb(\hmu_n)$, i.e., the optimal value of the OT barycenter, is more difficult to compute and its asymptotic limit under the null more difficult to sample from (\autoref{subsubsec:practical_diff}). Second, FDOTT can be immediately combined with any further post-hoc testing strategy, such as methods of simultaneous inference to identify individual significant factor levels, see \autoref{subsubsec:post_hoc}. Third, in contrast, the extension of the barycenter statistic to higher-order designs is not immediate (\autoref{sec:discussion}). As the barycenter method and FDOTT perform similarly in our (local) power analysis and in simulations, we finally suggest to use FDOTT, in general (see \autoref{subsec:comp}).
\end{remark}

\subsection{Theory: Main Results} \label{subsec:main_res}

As our main theoretical contribution, we derive the limit distributions of $\Tf(\hmu_n)$ and $\Tb(\hmu_n)$ under the null hypotheses $\nullf$ and $\nullb$, respectively, as well as under (local) alternatives. Our proof strategy extends \citet{Sommerfeld2018} to signed measures and more than two samples proving directional Hadamard differentiability of linear programs (LPs) for signed measures in conjunction with the delta method.

To this end, write $n_\star \defeq \min(n_1, \ldots, n_{\K})$ and assume that, as $n_\star \to \infty$, $\rho_n/n_k \to \delta_k \in [0,1]$ for all $k \in \idn{\K}$. Denote the multinomial covariance matrix $\Sigma(\mu^k) \in \R^{N \times N}$ by
\begin{equation} \label{eq:sigma}
  \Sigma(\mu^k)_{i,j} \defeq \begin{cases}
    \mu^k_i (1-\mu^k_i) & i = j\,, \\
    -\mu^k_i \mu^k_j & i \neq j\,,
  \end{cases} \qquad\text{ for } i,\, j \in \idn{N}\,.
\end{equation}
Let $G^k \sim \normaldist(0, \delta_k \, \Sigma(\mu^k))$, $k \in \idn{\K}$, be independent Gaussian random vectors and write
\begin{equation} \label{eq:gaussian}
  G \defeq [G^1, \ldots,  G^{\K}]^\transp\,.
\end{equation}
Based on a general theorem (\autoref{thm:limit_dist_fac}) for the empirical OT distance between signed measures, \autoref{thm:limit_dist_null_test_fac} asserts that under the null hypothesis $\nullf$ in \eqref{eq:null_fac} it holds that
\begin{equation} \label{eq:limit_dist_LR_null}
  \Tf(\hmu_n) \dconv \frac{1}{s}\sum_{m=1}^M \OText([L G]_m, \bbzero) \qquad\asntoinfty\,.
\end{equation}
The general limiting distribution of $\Tf(\hmu_n) - \Tf(\mu)$ (\autoref{thm:limit_dist_test_fac}), including that under the alternative and that for the OT barycenter test statistic $\Tb(\hmu_n)$ (see \autoref{sec:bary}), is a sum of random LPs, but with a more complicated structure. Furthermore, we derive  the limit distribution of $\Tf(\hmu_n)$ under \emph{local} alternatives as a Gaussian shift of the r.h.s.\ in \eqref{eq:limit_dist_LR_null} following the terminology of Le Cam, see e.g.\ \citep{LeCam2000}.

\begin{remark}[Extensions]\label{rem:extensions} Below we discuss potential extensions of our results.
  \begin{enumerate}
    \item (Dependent data). The samples $X_1^{1}, \ldots, X_{n_k}^k \sim \mu^k$, $k \in \idn{\K}$, are always assumed to be independent, for simplicity. This implies the joint convergence of the empirical probability vectors, i.e., $\sqrt{\rho_n} (\hmu_n - \mu) \dconv G$ as $n_\star \to \infty$, and serves as input for the delta method. Consequently, it is possible to allow the samples to have any dependency structure as long as $\sqrt{\lambda_n} (\hmu_n - \mu) \dconv H$ as $n_\star \to \infty$, where $\lambda_n \to \infty$ and $H$ is a random matrix in $\R^{\K \times N}$. In this case, our limit laws still hold but with $G$ and $\rho_n$ substituted by $H$ and $\lambda_n$, respectively. 
    
    For example, this becomes relevant when transforming (subsets of) $X_1^{1}, \ldots, X_{n_k}^k$, $k \in \idn{\K}$, which can lead to certain inter-sample dependencies. We demonstrate how this can be dealt with by applying FDOTT to distributions of distances in \autoref{subsec:dod}.
    \item (Countable spaces). We stress that it is possible to extend our asymptotic results for the FDOTT statistic from finitely to countably supported probability measures. Here, $\calX = \{x_1, x_2, \ldots\}$ is countable infinite and $c : \calX \times \calX \to \R_+$ a metric on $\calX$. For probability measures $\mu^1, \ldots, \mu^K$ on $\calX$ one has to assume that $\sum_{j = 1}^\infty c(x_j, x^*) \sqrt{\mu^k_{j}} < \infty$ for all $k \in \idn{\K}$ and some fixed point $x^* \in \calX$. To prove analogs to Theorems \ref{thm:limit_dist_fac}--\ref{thm:limit_dist_null_test_fac} and \ref{thm:local_limit} arguments in \citet{Tameling2019} and our strategy of proof can be extended, the details, however, are involved.
    \item (Continuous domain). Extending the analysis from countable to continuous domains such as $\R^d$ appears considerably more challenging. While the OT cost between probability measures is Hadamard directionally differentiable with respect to a sufficiently strong norm \citep{hundrieser2024unifying}, it is not clear how this property can be transferred to signed measures. The main obstacle is that, beyond the discrete setting, the positive and negative parts of signed empirical measures need not converge to the corresponding parts of the limiting signed measure. The detailed investigation offers an interesting route for future work.
    \item (Multiple testing and simultaneous inference). Our method can be extended to deal with multiple testing problems, i.e., when we want to simultaneously test a number of null hypotheses $H_0^{L_1}, \ldots, H_0^{L_q}$ of the form \eqref{eq:null_fac}. This is possible as we have access to the asymptotic multivariate distribution of the individual test statistics, and we showcase in \autoref{ex:tukey_hsd_test} an analog to Tukey's HSD test.
  \end{enumerate}
\end{remark}

\subsection{FDOTT in Action} \label{subsec:appl}

As a consequence of the limit law \eqref{eq:limit_dist_LR_null} for $\Tf(\hmu_n)$, the asymptotic level-$\alpha$ test FDOTT for the null hypothesis $\nullf$ rejects if
\begin{equation} \label{eq:fac_test}
  \Tf(\hmu_n) \geq q_{1-\alpha}\,, 
\end{equation}
where $q_{1-\alpha}$ is the $(1-\alpha)$-quantile of the limiting distribution \eqref{eq:limit_dist_LR_null}, $\alpha \in (0, 1)$. The quantile $q_{1-\alpha}$ can be approximated by the empirical quantile of $Z_1, \ldots, Z_J$ drawn from its asymptotic distribution \eqref{eq:limit_dist_LR_null}. Then, the $p$-value of FDOTT can be approximated by $\frac{1}{J} \#\{ j \in \idn{J} \mid \Tf(\hmu_n) \leq Z_j \}$.

\textbf{Software, Computation and Simulation Studies.} Computing $\Tf(\hmu_n)$ reduces to solving $M$ separate OT problems, hence standard OT solvers can be utilized for this task, see e.g.\ \cite{Peyre2019}. In contrast, approximating $q_{1-\alpha}$ via direct sampling from \eqref{eq:limit_dist_LR_null} is not straightforward as $\mu^1, \ldots, \mu^{\K}$ are unknown. To this end, we investigate a plug-in, $m$-out-of-$n$ bootstrap, derivative bootstrap and permutation approach (see \autoref{subsubsec:FDOTT_action}). In simulations (see \autoref{appendix:sim}), we found that the plug-in, derivative bootstrap and permutation approach have quite similar performance in terms of type I error and power, even for small sample sizes, where the performance of the $m$-out-of-$n$ bootstrap can be quite poor, depending on the choice of $m = o(n)$. Furthermore, drawing one sample using one of these schemes amounts to solving $M$ separate OT problems. This is the main computational bottleneck and, as such, the four approaches have very similar computational complexities. For very large sample sizes, note that $m$-out-of-$n$ bootstrap can have computational benefits due to sparseness of the bootstrap samples using suitable OT solvers. We found reasonable performance with $m = \sqrt{n}$. We point out that the plug-in, $m$-out-of-$n$ bootstrap and derivative bootstrap approach can be used for any higher-way layout, while it is not so obvious to extend the permutation approach beyond the one-way layout. All in all, we generally recommend using the plug-in or derivative bootstrap approach due to their performance and ease of use. In situations with small sample sizes and in the one-way layout, the permutation approach is a good alternative due to its non-asymptotic nature. Detailed practical guidelines are collected in \autoref{rem:practical_guidelines}.

All testing procedures introduced in this work are implemented in the R package \texttt{FDOTT} available at \url{https://cran.r-project.org/package=FDOTT}.

Finally, our simulations demonstrate that FDOTT has good power performance and is robust to spatial deformations. Exemplarily, we compare FDOTT to various different methods: rankFD \citep{Konietschke2023}, DISCO \citep{rizzo2010disco}, PERMANOVA \citep{Anderson2001} and KMD \citep{Huang2024}. In particular, in settings where the distributions differ in higher order moments, but not in means, we observe that FDOTT has superior power.

\textbf{Example from cellular biophysics.} We apply FDOTT to analyze the genetic knockout effect of vimentin on microtubules in the cytoskeleton of biological cells. Here, we provide a first illustration of our findings, for more details see \autoref{appendix:cell}. The cytoskeleton is an intricate network of three types of biopolymers; microtubules, actin filaments and intermediate filaments such as vimentin filaments: Microtubules function as the intracellular transport system, while actin filaments are responsible for force generation and cell migration, and intermediate filaments support the mechanical integrity of the cell, see e.g.\ \cite{Fletcher2010}. In \cite{Blob2024} it was investigated how actin filaments and vimentin intermediate filaments affect the structure of microtubules by comparing the cytoskeleton in different mouse embryonic fibroblasts. We considered two cell lines (denoted as factor I, $\K_1 = 2$ levels): MEF WT and NIH3T3. For both types, a genetically modified version is considered (factor II, $\K_2 = 2$ levels): Vimentin was knocked out (MEF VimKO, NIH3T3 VimKO) or not (MEF WT, NIH3T3). Cells from these four cell lines were then modified with respect to actin (factor III, $\K_3 = 2$ levels): No treatment (None) or actin and microtubules were disturbed via drugs (latrunculin A and nocodazole) and the microtubules were then regrown (LatA + Noc). Images of the cytoskeleton of these cells were recorded with fluorescence microscopy, see \autoref{fig:cells} for examples, and from these the microtubule structure was extracted. We model these structures by what we call microtubule histograms, $3$-dimensional histograms based on curvature, distance and angle to the cell center. Aggregating samples for each factor combination gives thus rise to a typical microtubule histogram for this combination. Overall, this results in a three-way layout, and we employ FDOTT (with plug-in approach, level of significance $\alpha = 0.05$) in combination with Tukey's HSD test to test for interaction and main factor effects, simultaneously. As the cost $c$ we take the Euclidean distance. For the resulting $p$-values see \autoref{tab:cells_anova}.

\textit{Interaction effects:} FDOTT does not provide any interaction effects between the three factors, i.e., the cell line as well as the presence/absence of vimentin intermediate filaments and actin filaments may influence the microtubules only separately. Therefore, we may analyze the effects of the three factors, separately.

\textit{Main effects:} The cell line seems to exhibit a significant main effect, while vimentin and actin do not. The latter is in accordance with the conclusions in \cite{Blob2024} that the presence/absence of vimentin intermediate filaments and actin filaments has no significant effect on the microtubules structure of the cytoskeleton, respectively. Further details of our analysis can be found in \autoref{appendix:cell}.

\begin{figure}
   \centering
    \includegraphics[width=0.24\textwidth]{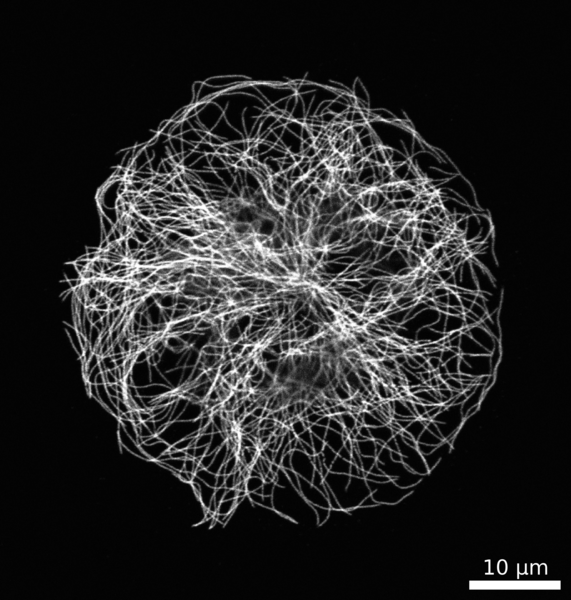}
    \includegraphics[width=0.24\textwidth]{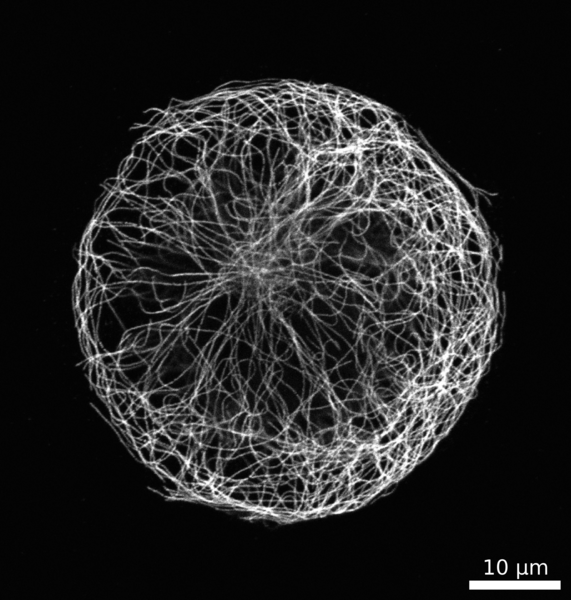}
    \includegraphics[width=0.24\textwidth]{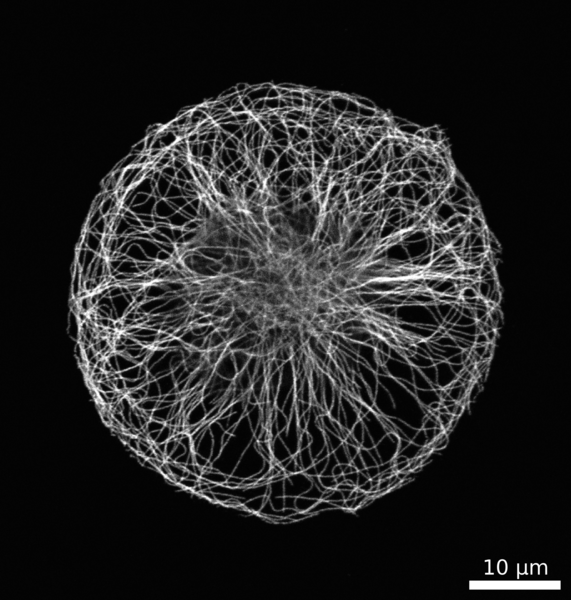}
    \includegraphics[width=0.24\textwidth]{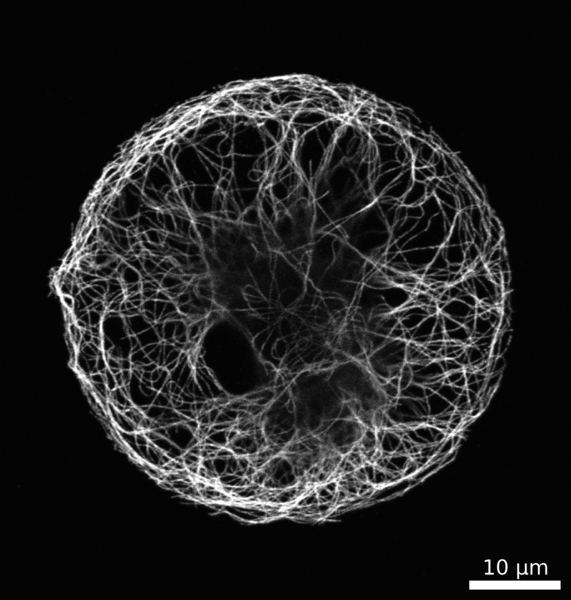}
     \caption{Images of the cytoskeleton of MEF cells. From left to right: (i) cell without any treatment (MEF WT), (ii) cell where vimentin is knocked out (MEF VimKO), (iii) cell where actin and microtubules were disturbed and microtubules were regrown (MEF WT + LatA + Noc), (iv) cell where vimentin is knocked out and actin and microtubules were disturbed and microtubules were regrown (MEF VimKO + LatA + Noc).}
   \label{fig:cells}
\end{figure}

\textbf{Fingerprint identification.} In \autoref{sec:fingerprints} we apply FDOTT to the discrimination of synthetic and real fingerprints. Interestingly, our methodology suggests the novel finding that its ability to imitate the minutiae patterns of the considered synthetic fingerprint generator significantly varies between so-called Henry classes.

\subsection{Related Work}

Parallel to our work we became aware of \cite{Kravtsova2025} who provides a test in the one-way layout for $\nullb$ based on multimarginal OT. As multimarginal OT and the OT barycenter problem are equivalent for suitably chosen cost functions (see e.g.\ \citealp[Section~6]{Carlier2010}), their test statistic is equivalent to the OT barycenter test statistic in the one-way layout and limited to this scenario (recall the discussion in \autoref{rem:barycenter_method}).

As already mentioned, our framework can be interpreted as an extension of nonparametric testing in factorial design. The corresponding tests are based on ranks, hence these methods implicitly assume an ordering of the observations while our approach allows to compare general (discrete) measures.

Another family of related methods is given by distance-based nonparametric ANOVA \citep{Anderson2001,Mueller2024,Dubey2019}. Roughly speaking, the idea is to emulate the right- or l.h.s.\ of \eqref{eq:idea} and replace the squared Euclidean distance by a squared metric. This works for any metric ground space and leads to test statistics based on the Fréchet mean \citep{Dubey2019} or pairwise distances \citep{Anderson2001,Mueller2024}. The former only works for the one-way layout while the latter is similar in spirit to FDOTT and can also deal with higher-way layouts. However, there is a fundamental difference in the statistical setting compared to ours: In \cite{Anderson2001,Mueller2024} the observed i.i.d.\ samples $Q_1^k, \ldots, Q_{n_k}^k \sim Q^k$, $k \in \idn{\K}$, are \textit{random} probability measures. In contrast, we assume that we observe samples $X_1^k, \ldots, X_{n_k}^k \sim \mu^k$, $k \in \idn{\K}$, from the underlying (unknown) probability measures $\mu^k \in \Delta_N$, $k \in \idn{\K}$. Hence, we do not ``directly'' observe such probability measures, rather empirical (random) probability vectors $\hmu_{n_1}^1, \ldots, \hmu_{n_{\K}}^{\K}$ that are based on them. It would be of great interest and practical relevance to combine both approaches which we plan to investigate in further work.

Lastly, we mention the OT-based ANOVA of \citet{masarotto2024transportation} to analyze second-order variation, i.e., if the covariance operators differ, in one-way layouts for functional data. This relies on a connection between the Procrustes metric between covariance operators and the OT between Gaussian processes \citep{masarotto2019procrustes}. In particular, this requires the formulation of the null hypotheses in terms of OT maps, which in general do not exist in our discrete setting.

\subsection{Organization}

This work is structured as follows: In \autoref{sec:testing}, we develop the required theoretical foundation and derive the limit distributions of FDOTT (\autoref{subsec:fac_design}). In \autoref{sec:fingerprints}, we apply our methods to two real world datasets. Finally, in \autoref{sec:discussion} we discuss extensions and further research directions. The appendix contains all proofs (\autoref{appendix:proofs}), additional details regarding the barycenter method (\autoref{sec:bary}), real data examples (\autoref{appendix:cell}), and comprehensive numerical simulations and comparisons with other methods (\autoref{appendix:sim}).

\section{Theory: Limit Distributions} \label{sec:testing}

Recall our setting in \autoref{subsec:approach}. In this section, we derive limit laws for the FDOTT statistic $\Tf(\hmu_n)$ in \eqref{eq:test_stat_fac} and employ it to introduce several statistical tests.

\subsection{Factorial Design} \label{subsec:fac_design}

Fundamental to the limit distribution of the FDOTT statistic $\Tf(\hmu_n)$ in \eqref{eq:test_stat_fac} is our main result.

\begin{theorem}\label{thm:limit_dist_fac}
  Recall the setting in \autoref{subsec:approach} and \eqref{eq:ot_ext}. For $\tau\in\Delta_N^{\pm}$, define the set of OT dual solutions
  \begin{equation} \label{eq:Phi}
  \Phi^*(\tau):=\bigl\{(u,v)\in\R^{N + N} : \begin{array}{l}
  \inner{u}{\tau^+}+\inner{ v}{\tau^-}=\OT_c(\tau^+,\tau^-), \;
  u \oplus v \leq c
  \end{array}\bigr\}.
\end{equation}
Here, $u \oplus v = [u_i + v_j]_{i,j=1}^N$ denotes the outer sum. Define for $A,\,B \in \R^{M \times N}$
\begin{equation*}
  U_{A}(B) \defeq B \odot [ \I(A > 0) + \I(A = 0, B > 0) ],\;\; V_{A}(B) \defeq -B \odot [ \I(A < 0) + \I(A = 0, B < 0) ],
\end{equation*}
where $\I(\argdot)$ is the indicator function and $\odot$ the componentwise (Hadamard) multiplication. Then, with $\rho_n$ as in \eqref{eq:rho} and $G$ in \eqref{eq:gaussian}, as $n_\star\rightarrow\infty$,
\begin{align*}
&\sqrt{\rho_n} \biggl[
\OText((L\hmu_n)_m, \bbzero) - \OText((L\mu)_m,\bbzero) \biggr]_{m \in \idn{M}} \\
\dconv &\left[\max_{(u_m,v_m)\in\Phi^*([L\mu]_m)}
\inner{u_m}{U_{L\mu}(LG)_m} + \inner{v_m}{V_{L\mu}(LG)_m} \right]_{m \in \idn{M}}.
\end{align*}
\end{theorem}

As a consequence, we obtain the limit law for the FDOTT statistic $\Tf(\hmu_n)$.

\begin{theorem} \label{thm:limit_dist_test_fac}
  Under the setting of \autoref{thm:limit_dist_fac} with $s > 0$ in \eqref{eq:test_stat_fac}, as $n_\star \to \infty$,
  \begin{equation*}
    \Tf(\hmu_n) - \Tf(\mu) \dconv \frac{1}{s} \sum_{m=1}^M \left(\max_{(u_m,v_m)\in\Phi^*([L\mu]_m)} \inner{u_m}{U_{L\mu}(LG)_m} + \inner{v_m}{V_{L\mu}(LG)_m} \right)\,.
  \end{equation*}
\end{theorem}

The following result shows that the limit simplifies significantly under the null $\nullf$ in \eqref{eq:null_fac} as then $\Phi^*(\bbzero) = \{ (u, v) \in \R^{N + N} : u \oplus v \leq c \}$ is the whole set of the dual variables of the OT problem.

\begin{theorem} \label{thm:limit_dist_null_test_fac}
  Suppose that the null hypothesis $\nullf$ in \eqref{eq:null_fac} holds. Then, under the setting of \autoref{thm:limit_dist_fac}, as $n_\star \to \infty$,
  \begin{equation*}
    \Tf(\hmu_n) \dconv \frac{1}{s} \sum_{m=1}^M \OText([LG]_m, \bbzero)\,.
  \end{equation*}
\end{theorem}

\begin{remark}[Normality] \label{rem:fac_normality}
  Note that the limiting distribution in \autoref{thm:limit_dist_test_fac} is Gaussian if $\Phi([L\mu]_m)$ is a singleton for all $m \in \idn{M}$. The latter occurs if $[L\mu]^+_m$ and $[L\mu]^-_m$ fulfil the so-called primal summability condition, see \cite{Klee1968} or \cite{staudt2025uniqueness}. However, under the null hypothesis $\nullf$ the limit is never Gaussian as the primal solution is unique and degenerate \citep[Fact~2.4~iii)]{Klatt2022}.
\end{remark}

\subsection{Factorial Design: Examples}

We give an explicit elaboration of the FDOTT statistic in the one- and two-way layout.

\begin{example}[One-Way Layout] \label{ex:one_way_1}
  Consider the null hypothesis $H_0: \mu^1=...=\mu^{\K}$ and bring this in the form of \eqref{eq:null_fac}, 
  \begin{equation*}
    \nullf : \mu^{i} - \mu^{j} = \bbzero \quad \text{ for all $1 \leq i <  j \leq \K$.}
  \end{equation*}
  Here, $M=\K(\K-1)/2$ is the number of pairwise comparisons to be made, and the FDOTT statistic with $s = \K^2$ reduces to
  \begin{equation*}
    \Tf(\hmu_n) = \frac{\sqrt{\rho_n}}{\K^2}\sum_{i<j}\OText(\hmu_{n_i}^i - \hmu_{n_j}^j, \bbzero)\,.
  \end{equation*}
  Under the null hypothesis it holds that (\autoref{thm:limit_dist_null_test_fac})
  \begin{equation} \label{eq:limit_dist_null_one_way}
    \Tf(\hmu_n) \dconv \frac{1}{\K^2} \sum_{i < j} \OText(G^i-G^j, \bbzero)\,, \qquad\asntoinfty\,.
  \end{equation}
  Note, that as $c$ is a metric cost matrix, it holds that $\OText(\hmu_{n_i}^i - \hmu_{n_j}^j, \bbzero) = \OText(\hmu_{n_i}^i, \hmu_{n_j}^j)$, see \cite[Remark~3.5]{Mainini2012}. In particular, it follows that
  \begin{equation*}
    \Tf(\hmu_n) = \frac{\sqrt{\rho_n}}{\K^2}\sum_{i<j}\OT_c(\hmu_{n_i}^i, \hmu_{n_j}^j)\,,
  \end{equation*}
  in close analogy to the one-way ANOVA test statistic in \eqref{eq:idea}.
\end{example}

\begin{example}[Two-Way Layout] \label{ex:two_way_1}
  Recall the two-way layout in \eqref{eq:interactionAB}. Use the same notation for the means as in \eqref{eq:mu_mean} for $\hmu_n$ and $G$. Then, the FDOTT statistic for the interaction effect between factor A and B with $s=\K_1\K_2$ is given by
  \begin{equation*}
    \Tf(\hmu_n) = \frac{\sqrt{\rho_n}}{\K_1 \K_2}\sum_{i,j=1}^{\K_1,\K_2}\OText(\hmu^{ij}_n - \hmu^{i\bullet}_n - \hmu^{\bullet j}_n + \hmu^{\bullet\bullet}_n, \bbzero)\,.
  \end{equation*}
  Under the null hypothesis, it holds that
  \begin{equation*}
    \Tf(\hmu_n) \dconv \frac{1}{\K_1\K_2} \sum_{i,j=1}^{\K_1,\K_2} \OText(G^{ij} -  G^{i\bullet} - G^{\bullet j} + G^{\bullet\bullet}, \bbzero)\,, \qquad\asntoinfty\,.
  \end{equation*}
  If no interaction effect is present, the effects of the first and second factor can be analyzed separately. For example, the null hypothesis of no main effect for A is given by
  \begin{equation*}
    H_0 : \mu^{i\bullet} - \mu^{\bullet\bullet} = \bbzero \quad \text{ for all } i \in \idn{\K_1}\,.
  \end{equation*}
  If there is an interaction effect, so-called simple factor effects might be of interest. The null hypothesis for no simple factor effect for A within B reads
  \begin{equation*}
    H_0 : \mu^{ij} - \mu^{\bullet j} = \bbzero \quad \text{ for all } i \in \idn{\K_1},j \in \idn{\K_2}\,.
  \end{equation*}
  Similarly, we can define the main effect of B and simple factor effect of B within A. The test statistics and limit distributions for main and simple factor effects follow from \autoref{thm:limit_dist_null_test_fac} analogously to the interaction effect and are omitted here.
\end{example}

\section{FDOTT in Action}

\subsection{Practical Implementation} \label{subsubsec:FDOTT_action}

As the measures $\mu^1,...,\mu^{\K}$ are unknown direct sampling from the limiting distribution of $\Tf(\hmu_n)$ is impossible. Instead, we propose and compare the three following approaches which are compared to other methods in a simulation study in \autoref{appendix:sim}.

\textbf{Plug-in.} Substitute $\hmu_{n_1}^{1}, \ldots, \hmu_{n_{\K}}^{\K}$ for $\mu^1,...,\mu^{\K}$ in the limiting distribution given in \autoref{thm:limit_dist_null_test_fac}, i.e., plug in $\hG^k_{n_k} \sim \normaldist(0, \Sigma(\hmu^k_{n_k}))$ instead of $G^k \sim \normaldist(0, \Sigma(\mu^k))$. Call $\hq_{1-\alpha,n}$ the random $(1-\alpha)$-quantile of the resulting limit distribution. We propose to substitute $q_{1-\alpha}$ by $\hq_{1-\alpha,n}$ in \eqref{eq:fac_test}. It immediately follows by Slutzky's theorem that this test keeps its level $\alpha$ asymptotically. Further, the test is consistent.

\begin{theorem} \label{thm:fac_test_power}
  Under any fixed alternative of $\nullf$, it holds that $\lim_{n\to\infty} \Pr{\Tf(\hmu_n) \geq \hq_{1-\alpha,n}} = 1$.
\end{theorem}

\begin{remark}[Pooling] \label{rem:enlargening}
  Note that for some test designs it is possible to ``enlarge'' the sample sizes used to obtain a more accurate approximation of the quantile, i.e., to calculate $\hq_{1-\alpha,n}$. For example, in the one-way layout, under the null hypothesis $\mu^1=\ldots=\mu^{\K}$, i.e., all samples $X^k_1,...,X^k_{n_k}, k \in \idn{\K}$, come from the same distribution. Therefore, we can use all $\sum_{k=1}^{\K} n_k$  samples to estimate $\mu^1=\ldots=\mu^{\K}$. Note that this pooling procedure is compatible with \autoref{thm:fac_test_power}. While this improves the finite sample accuracy of the nominal level this can decrease the power though, see the simulations in \autoref{appendix:sim}.
\end{remark}

\textbf{Bootstrap.} To (approximately) sample from the limiting distribution of $\Tf(\hmu_n)$, we can also employ bootstrap procedures \citealp{Duembgen1993, Fang2018}. Due to the shared proof strategy, the statements from \citet[Appendix~B]{Sommerfeld2018} apply mutatis mutandis here: The classical $n$-out-of-$n$ bootstrap is not consistent, while the $m$-out-of-$n$ bootstrap for $m = o(n)$, as well as the derivative bootstrap are both consistent.

For the $m$-out-of-$n$ bootstrap, to avoid confusion call the bootstrap sample size by $\ell$ instead of $m$. First, we sample i.i.d.\ random variables $X_1^{k,*}, \ldots, X_{\ell_k}^{k,*} \sim \hmu_{n_k}^k$, $k \in \idn{\K}$, compute the corresponding empirical probability vectors $\hmu^{k,*}_{\ell_k}$ and set $\hmu_{\ell}^{*} \defeq [ \hmu^{1,*}_{\ell_1}, \ldots, \hmu_{\ell_{\K}}^{\K,*}]^\transp$. One bootstrap sample to imitate the limiting distribution is then given by $ \frac{\sqrt{\rho_\ell}}{s} \sum_{m=1}^M \OText ( [L\hmu^{*}_\ell]_m, \bbzero)$.

For the derivative bootstrap, we generate a bootstrapped version $\hmu^*_{n}$ of $\hmu_n$ as above (with $\ell_k = n_k$, $k \in \idn{\K}$). A bootstrap sample can then be computed by $\frac{1}{s} \sum_{m=1}^M \OText( [L \sqrt{\rho_n} (\hmu^*_{n} - \hmu_n)]_m, \bbzero)$.

\textbf{Permutation.} In the one-way layout (\autoref{ex:one_way_1}), under the null hypothesis the underlying measures $\mu^1, \ldots, \mu^{\K}$ are all equal, hence the samples are invariant under permutations, which allows a random permutation approach, see e.g.\ \cite{Hemerik2018,Lehmann2022}. Let $X^{1,*}_{1}, \ldots, X^{1,*}_{n_{1}}$, $\ldots$, $X^{\K,*}_{1}, \ldots, X^{\K, *}_{n_{\K}}$ be a random (joint) permutation of all the samples $X_1^1, \ldots, X^1_{n_1}$, $\ldots$, $X^{\K}_{1}, \ldots, X^{\K}_{n_{\K}}$. Constructing the corresponding $\hmu^*_{n}$ from these, a permutation sample is given by $ \frac{\sqrt{\rho_n}}{s} \sum_{m=1}^M \OText ( [L\hmu^{*}_n]_m, \bbzero)$. Let $Z_1, \ldots, Z_J$ be samples created this way, then the $p$-value of FDOTT can be approximated by $\frac{1}{J+1} [\#\{ j \in \idn{J} \mid \Tf(\hmu_n) \leq Z_j \} + 1]$. This ensures that the level of the corresponding test is upper bounded by $\alpha$, see \citet[Equation~(17.8)]{Lehmann2022}.

All these tests are implemented in the R package \texttt{FDOTT} available at \url{https://cran.r-project.org/package=FDOTT}.

\begin{remark}[Practical guidelines] \label{rem:practical_guidelines}
  Based on our simulations in \autoref{appendix:sim}, we recommend the following. The plug-in and the derivative bootstrap approach behave almost identically in terms of level and power, are applicable in any layout and are of comparable computational cost; we thus suggest the plug-in approach as the default. For small sample sizes in the one-way layout, the permutation approach can be preferable, as it performs on par with them while controlling the level non-asymptotically; in higher-way layouts it is not available. The $m$-out-of-$n$ bootstrap is the least robust choice, since its performance hinges critically on $m$: for $m$ too close to $n$ the test hardly ever rejects, whereas $m = \sqrt{n}$ yields a reasonable, though slightly liberal (one-way) or conservative (two-way) test. We hence advise using it only for reasonably large $n$, where the sparsity of the bootstrap samples pays off computationally. Finally, a large number of measures $\K$ or a large ground space slow down the convergence to the limit law and raise the cost per draw ($M$ grows quadratically in $\K$ in the one-way layout), so that correspondingly larger sample sizes are needed; if these are not available, pooling can be useful (\autoref{rem:enlargening}).
\end{remark}

\subsection{Post-Hoc Testing} \label{subsubsec:post_hoc}

In applications of the conventional ANOVA method, post-hoc testing strategies can be incorporated when the ANOVA test rejected the global null hypothesis. This is done to identify individual significant deviations between the factor levels. Here, this translates to considering which of the sub-hypotheses $\nullf[m] :[L\mu]_m=\bbzero$, $m \in \idn{M}$, are violated. Of course, each of these hypotheses can be tested separately via techniques introduced so far and a variety of well-known multiplicty adjustments for post-hoc testing can also be applied here, see e.g. \cite{Miller1981,Benjamini1995,Dickhaus2014}. However, it is well known that methods which directly target simultaneous inference statements can be more powerful. We will illustrate this exemplarily for the analog to Tukey's honestly significant difference (HSD) test, see e.g.\ \citet[Chapter~2]{Miller1981}.

\begin{example}[Tukey's HSD Test] \label{ex:tukey_hsd_test}
Consider the maximum (instead of the sum) of the differences $\Tf[m](\hmu_n) \defeq \sqrt{\rho_n} \OText([L\hmu_n]_m, \bbzero)$ in \eqref{eq:test_stat_fac}. Under the null hypothesis $\nullf$, \autoref{thm:limit_dist_fac} yields (as the maximum is continuous) that
\begin{equation} \label{eq:limit_dist_hsd}
\max_{m\in\idn{M}} \Tf[m](\hmu_n) \dconv \max_{m \in \idn{M}} \OText([LG]_m, \bbzero) \qquad \asntoinfty\,.
\end{equation}
Tukey's HSD test rejects $\nullf[m]$ for each $m \in \idn{M}$ with $\Tf[m](\hmu_n) \geq t_{1-\alpha}$, the $(1-\alpha)$ quantile of the r.h.s.\ of \eqref{eq:limit_dist_hsd}. Hence, the probability of at least one erroneous rejection, is bounded by $\alpha$.
\end{example}

\begin{example}[Weighted Turkey's HSD Test] \label{ex:tukey_hsd_test_w}
For the one-way layout (\autoref{ex:one_way_1}), the limit law used for Tukey's HSD test can be rewritten as
\begin{equation*}
  \max_{i<j} \sqrt{\rho_n} \OT_c(\hmu^i_{n_i},\hmu^j_{n_j}) \dconv \max_{i < j}\left(\max_{u_{ij}} \inner{ u_{ij}}{G^i-G^j}\right)\,, \qquad \asntoinfty\,,
\end{equation*}
where the inner maximum is taken over all $u_{ij} \in \R^{N}$ such that $u_{ij} \oplus (-u_{ij}) \leq c$. Note that the maximum in the limiting distribution is dominated by larger sample sizes $n_i$ (in relation to $n_{j}$, $i \neq j$) as then the weights $\delta_i$ are bigger. We propose to reduce this effect by using properly chosen weights $[w_{ij}]_{i<j}$ with $w_{ij}>0$ and $\sum_{i<j}w_{ij}=1$. To this end, observe that $\Var(w_{ij}\inner{u_{ij}}{G^i-G^j}) = w_{ij}^2(\delta_i+\delta_j)u_{ij}^\transp\Sigma(\mu^1)u_{ij}$. Choosing $w_{ij}=\frac{1}{\sqrt{\delta_i+\delta_j}} (\sum_{r<s}\frac{1}{\sqrt{\delta_r+\delta_s}})^{-1}$ equalizes the prefactor of the variances and, to some extent, eliminates the sample size weighting in the limiting distribution. Denote with $\hat{w}_{n,ij}$ the version of $w_{ij}$ where $\delta_r$ is substituted by $\rho_ n / n_r$. Then, we can perform the weighted Tukey's HSD test based on
\begin{equation*}
  \max_{i<j} \sqrt{\rho_n} \hat{w}_{n,ij} \OText(\hmu^i_{n_i} - \hmu^j_{n_j}, \bbzero) \dconv \max_{i < j} w_{ij} \OText(G^i - G^j, \bbzero) \qquad \asntoinfty\,.
\end{equation*}
Simulations (see \autoref{appendix:sim}) demonstrate that this weighting can indeed increase the performance of the test.
\end{example}

\subsection{Summary of Simulation Study}

We compare the different FDOTT procedures described in \autoref{subsubsec:FDOTT_action} as well as similar methods in simulations in \autoref{appendix:sim}. In the one-way layout, we compare FDOTT with the barycenter method (\autoref{rem:barycenter_method}), rankFD \citep{Konietschke2023}, DISCO \citep{rizzo2010disco}, PERMANOVA \citep{Anderson2001} and KMD \citep{Huang2024} as these reflect a broad range of different methodologies, including rank, distance and kernel based tests. We consider two settings, one where the probability vectors are given by (truncated and then normalized) Poisson probabilities with varying parameter $\lambda > 0$, and one where the means of the distributions are the same but the covariances differ. 
In the first setting our simulations showcase that all the above methods have similar performance (except rankFD which can have significantly higher power in some cases). The second setting reveals that FDOTT is more sensitive to changes in covariance than the other methods (with the exception of KMD, depending on the choice of a hyperparameter) which results in superior power. Furthermore, in the two-way layout, we compare FDOTT with rankFD, DISCO and PERMANOVA. There, we see that our notion of interaction effect is not compatible with DISCO and PERMANOVA, i.e., they tend to reject the corresponding null hypothesis although true. Additionally, FDOTT seems to have superior power compared to rankFD.

For the comparison of the different FDOTT procedures, see \autoref{rem:practical_guidelines}.

\section{Application: Real and Synthetic Fingerprints} \label{sec:fingerprints}

In the following, we apply FDOTT to analyze real and synthetic fingerprints. Exemplarily, we demonstrate how to extend FDOTT to compare distributions of distances which are estimated from dependent samples (recall \autoref{rem:extensions}.1).

\subsection{FDOTT for Distribution of Distances} \label{subsec:dod}

Recall the underlying ground space $\calX = \{x_1, \ldots, x_N\}$ and suppose that we have a symmetric function $d : \calX \times \calX \to \{d_1, \ldots, d_R\}$ taking $R$ distinct values, which we view as some kind of ``dissimilarity measure'', e.g., $d$ could be a distance on $\calX$. This allows us to consider the \textit{distributions of distances} $\nu^k \defeq d_{\#} [\mu^k \otimes \mu^k]$, $k \in \idn{\K}$, instead of the underlying $\mu^1, \ldots, \mu^{\K}$. Here, $\nu^k$ is the pushforward measure of the independent coupling $\mu^k \otimes \mu^k$ by $d$, which we can view as a probability vector of size $R$,
\begin{equation*}
  \nu^{k}_r = \sum_{p, q = 1}^N \I( d(x_p, x_q) = d_r ) \mu_p \mu_q\,, \qquad  r \in \idn{R},\, k \in \idn{\K}\,.
\end{equation*}
Writing $\nu \defeq [\nu^1, \ldots, \nu^{\K}]^\transp$, instead of $\nullf$ in \eqref{eq:null_fac} we can consider null hypotheses of the form $H_0^{\mathrm{DoD},L} : L\nu = \bbzero$. Estimating $\nu^k$ based on i.i.d.\ $X_1^k, \ldots, X_{n_k}^k \sim \mu^k$ by
\begin{equation*}
\hnu^k_{n_k, r} = \frac{2}{n_k (n_k - 1)} \sum_{1 \leq p < q \leq n_k} \I( d(X^k_p, X^k_q) = d_r ) \,, \qquad r \in \idn{R}\,, k \in \idn{\K}\,,
\end{equation*}
we can test for such hypotheses again using the FFDOTT statistic $\Tf(\hnu_n)$ in \eqref{eq:test_stat_fac} where $\hnu_n \defeq [\hnu^1_{n_1}, \ldots, \hnu^{\K}_{n_{\K}}]^\transp$ and the cost matrix $c$ must now be of size $R \times R$. However, due to the underlying dependency of the pairwise distances, \autoref{thm:limit_dist_null_test_fac} cannot immediately be used to obtain the limiting distribution of $\Tf(\hnu_n)$. Exploiting the $U$-statistic structure of $\hnu_{n_k}^k$, in \autoref{lemma:dod_dist} we show that $\sqrt{\rho_n} [\hnu_n - \nu]$, as $n_\star \to \infty$, is still normally distributed from which we can obtain the limit law for $\Tf(\hnu_n)$ (as discussed in \autoref{rem:extensions}.1). This serves as the theoretical basis for our data analysis in the following section.

\subsection{Comparison of Real and Synthetic Fingerprints}

In this section, we apply FDOTT to discriminate synthetic and real fingerprints (``type'') which is a challenging endeavor in biometric identification (see \citealp{Maltoni2009}). The major characteristic in fingerprint identification are minutiae, i.e., endings and bifurcations of the ridges of a fingerprint \citep{Jain2007}. Any minutia can be described by a location (2D coordinate) and direction (angle), see \autoref{fig:fingerprint} for an example of a real and synthetic fingerprint together with their minutiae. A minutiae histogram contains the information of pairs of minutiae \citep{Gottschlich2014}: Each pair is binned by their distance and the difference between their angles. Note that this leads to a rotationally invariant statistic of the data. Aggregating these histograms for several fingerprints of the same type gives a typical minutiae histogram for that type. These were then used by \citet{Sommerfeld2018} to assess statistically significant differences between real and synthetic fingerprints. We extend and refine this two-sample analysis for so-called Henry classes, see \citet{Maltoni2009} and \autoref{fig:fingerprint} for examples. In particular, \citet{Sommerfeld2018} had to take randomly chosen disjoint pairs of minutiae (hence potentially reducing the power of their test) from each fingerprint to avoid that pairs are dependent. Now, we can use our theory for distributions of distances (see \autoref{subsec:dod}) to explicitly deal with the dependency.

In the following, we only consider the three most common classes, i.e., left loop, right loop and whorl. This leads to a two-way layout with synthetic/real (factor I, $\K_1 = 2$ levels) times Henry classes (factor II, $\K_2 = 6$ levels). To perform FDOTT, we use databases 1 (real fingerprints) and 4 (synthetic fingerprints) of the Fingerprint Verification Competition of 2002 \citep{Maio2002}. We assigned the Henry class of each fingerprint manually and kept only the ones with classes left loop, right loop and whorl. For each fingerprint, minutiae were extracted by an automatic procedure using a commercial off-the-shelf program (VeriFinger 5.0 by \citealp{VeriFinger}). We centered each fingerprint by the mean minutiae location and then scaled each minutia location to the two-dimensional unit cube $[0, 1]^2$ which we discretize into $15 \times 15$ equidistant bins. Similarly, we discretized the minutia direction by splitting $[0, 2\pi)$ into $18$ bins. Then, we aggregated all minutiae of fingerprints of the same type, see \autoref{tab:fingerprints_samplesizes} for the corresponding sample sizes. Thus, each minutia corresponds to one sample on the $15 \times 15 \times 18$ grid.

\begin{figure}
  \centering
  \includegraphics[height=5cm]{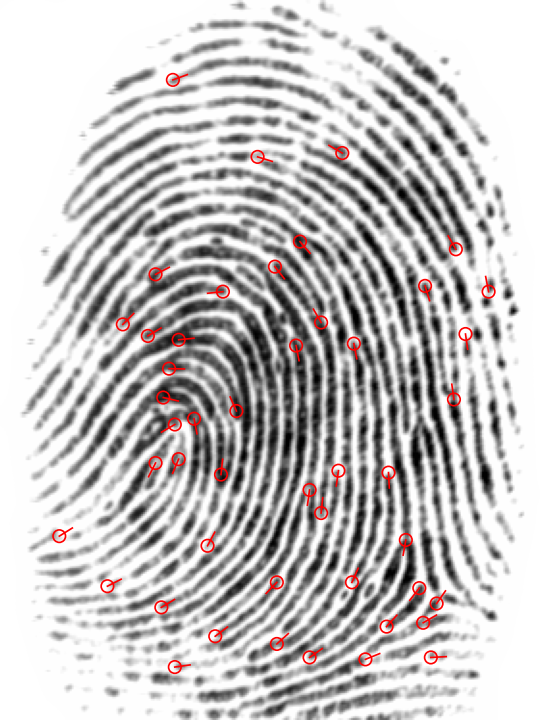}
  \includegraphics[height=5cm]{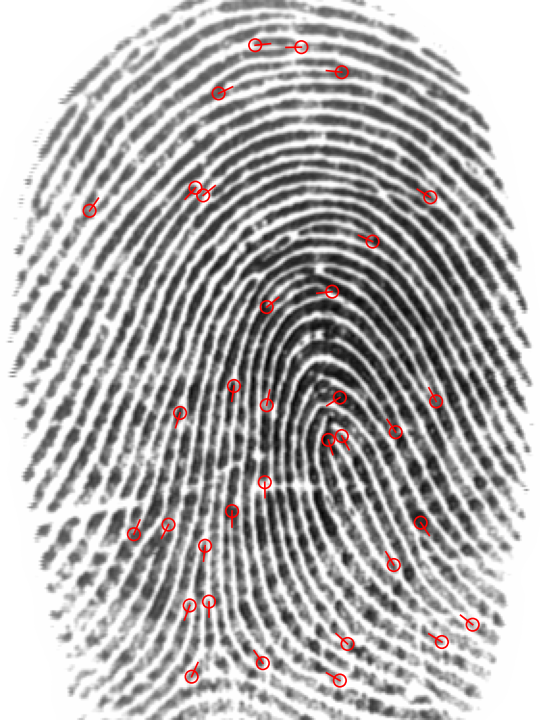}
  \includegraphics[height=5cm]{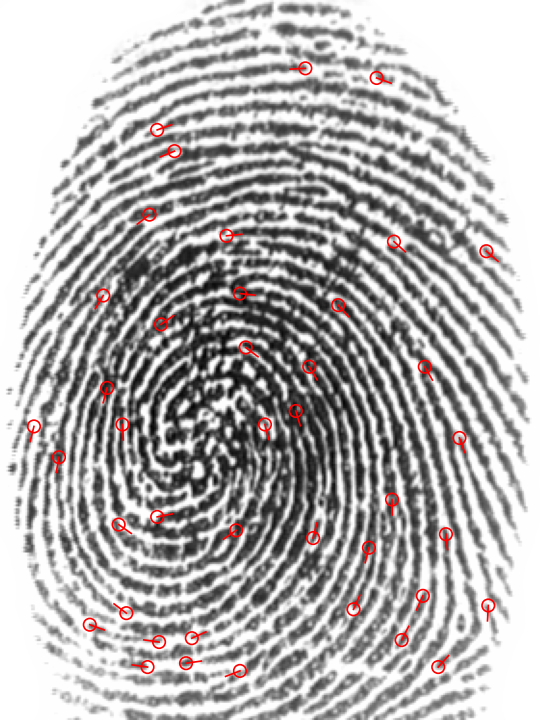}
  \includegraphics[height=5cm]{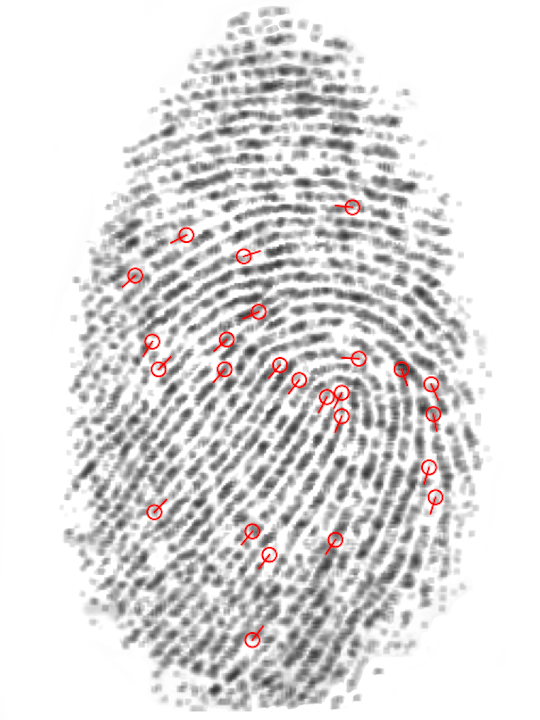}
  \includegraphics[height=5cm]{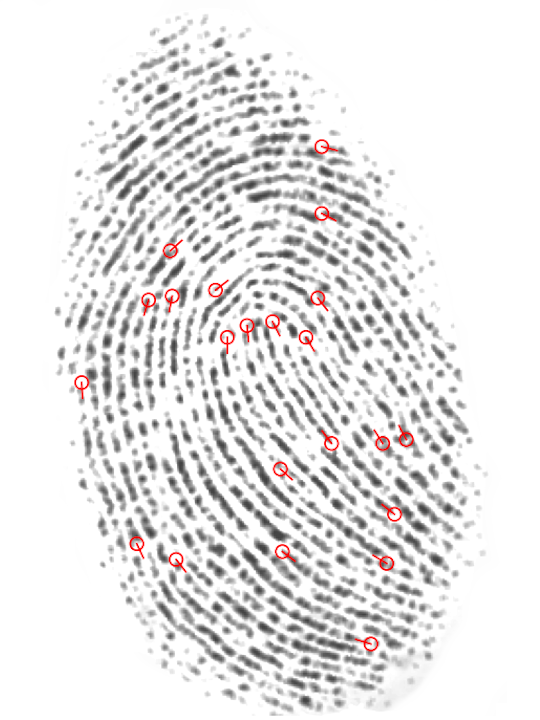}
  \includegraphics[height=5cm]{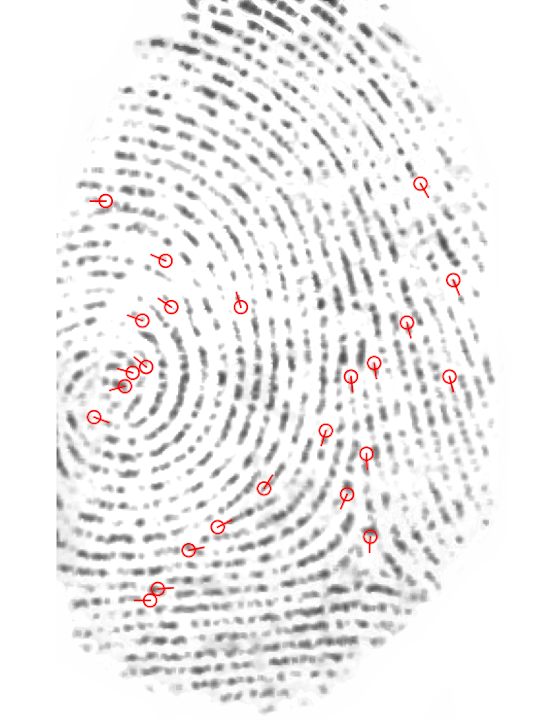}
  \caption{Pictures of real (top row) and synthetic (bottom row) fingerprints and their minutiae (indicated by the red circles). Their Henry classes are, from left to right: ``left loop'',  ``right loop'' and ``whorl''.}
  \label{fig:fingerprint}
\end{figure}

Using FDOTT for the distribution of distances as detailed in \autoref{subsec:dod}, we analyze the interaction and main effects of the two factors. In a first step, we check whether all the fingerprint distance distributions are the same (one-way layout).

As the distance between $x = (x_1, x_2, x_3)$ and $y = (y_1, y_2, y_3)$ on the $15 \times 15 \times 18$ grid, we took $d(x, y) \defeq ( \abs{x_1 - y_1} + \abs{x_2 - y_2},\, \abs{x_3 - y_3} )$, i.e., we calculated the location and angle distance separately. Note that the resulting distributions of distances all lie on the regular $29 \times 18$ grid, and we take the Euclidean distance between the points on the grid as the underlying cost matrix $c$. To calculate the quantile $q_{1-\alpha}$, we simulated 1000 values from the limiting distribution via the plug-in approach. All results are significant at significance level of $\alpha = 0.05$. First, FDOTT in the one-way layout ($p < 10^{-3}$) asserts that there are differences in the minutiae histograms for all factor combinations. Furthermore, we find significant interaction and main effects (both $p < 10^{-3}$) for the Henry classes and the type of fingerprint (real/synthetic) which suggests that the fake quality of synthetic fingerprints depends on the Henry classes. As we observe a significant interaction effect, the (significant) main factor effects are difficult to interpret. Therefore, we also consider the simple factor effects. We observe significant simple factor effects (both $p < 10^{-3}$), which suggests that both factors have an effect on the minutiae histograms.

Since the test comparing all 6 categories at once rejected the null hypothesis (one-way layout), we also performed Tukey's HSD test as described in \autoref{ex:tukey_hsd_test} to analyze where exactly the differences lie. We performed the method twice, once without and once with weighting. The results are summarized in \autoref{tab:fingerprints_tukey}. First, we see that the differences within all real or synthetic fingerprints are significant, i.e., the three Henry classes for real (resp.\ synthetic) fingerprints are well-separated by their minutiae histograms. Moreover, there is only one non-significant differences between real/synthetic: (real, whorl) vs.\ (synthetic, whorl). This suggests that the synthetic fingerprint generator is able to create realistic fingerprints for the Henry class ``whorl''. Interestingly, our analysis demonstrates that its seems to struggle with the Henry classes ``left loop'' and ``right loop''.

\begin{table}
  \centering
  \begin{tabular}{cc|rr|rrr}
    & & real & & synthetic & & \\
    & & right loop & whorl & left loop & right loop & whorl \\\hline
    real & left loop &  $< 10^{-3}*$ & $< 10^{-3}*$ & $< 10^{-3}*$ & $< 10^{-3}*$ & $< 10^{-3}*$  \\
     & right loop  &  & $< 10^{-3}*$ & $< 10^{-3}*$ & $0.002*$ & $0.002*$ \\
    & whorl  &  &  & $< 10^{-3}*$ & $< 10^{-3}*$ &  $0.666\phantom{*}$ \\\hline
    synthetic & left loop  &  &  &  & $< 10^{-3}*$ & $< 10^{-3}*$ \\
    & right loop  &  &  &  &  & $< 10^{-3}*$ \\
  \end{tabular}
  \caption{$p$-values of Tukey's HSD test for the fingerprint minutiae histograms in the one-way layout without weighting. Significant $p$-values for $\alpha = 0.05$ are marked by a ``$*$''. Performing the test with weighting leads to the same significant differences.}
  \label{tab:fingerprints_tukey}
\end{table}

\section{Discussion} \label{sec:discussion}

We close with a discussion of future research directions and extensions of our work.

It might be of interest to consider a variant of OT, e.g., with further computational benefits. For instance, promising candidates might be the debiased Sinkhorn divergence \citep{Feydy19} or sliced Wasserstein distances \citep{Bonneel2015}.

Further, recall that the permutation approach as given in \autoref{subsubsec:FDOTT_action} only works in the one-way layout. It is of interest to extend this to higher-way layouts, e.g., as done in \citet{Pauly2014} for Wald-type statistics in general factorial designs. As FDOTT is not directly based on the samples, but on the corresponding empirical probability vectors, the techniques of \citet{Pauly2014} are not immediately applicable, which is an interesting direction for further work. Similarly, it would be of interest to obtain an exact permutation procedure in spirit of \citet{Hemerik2018} which, however, is far beyond the scope of this paper.

Throughout, we have assumed that the underlying cost matrix $c$ is a metric which implies that $\OText$ and $\OT$ are both metrics. However, for FDOTT to work, it suffices that the cost functional allows to identify measures from $\Delta_N^{\pm}$ and $\Delta_N$, respectively. As such, it is enough that $c$ is also identifiable, i.e., $c \geq 0$ and $c_{ij} = 0$ if and only if $i = j$ for all $i, j \in \idn{N}$. For instance, this is the case if $c$ is a metric cost matrix to the (componentwise) power of $p > 0$. Nevertheless, there are instances in applications where the cost matrices are not identifiable, and FDOTT therefore has to be applied with some care. One example is the comparison of gene expressions where the cost matrices are correlation matrices \citep{nies2025transport}.

Lastly, recall that the barycenter method can only be used for the one-way layout in contrast to FDOTT. This could be overcome by extending the OT barycenter to signed measures via $\OText$. However, due to the non-linear nature of the Jordan decomposition, the resulting $\OText$-barycenter functional cannot in general, unlike the $\OT$-barycenter functional $B_c^w$, be rewritten as an LP. Hence, the techniques used here cannot be applied to deal with this extension.

\section*{Acknowledgments}

A.M.\ gratefully acknowledge support of the Deutsche Forschungsgemeinschaft (DFG, German Research Foundation) FOR 5381 ``Mathematical Statistics in the Information Age'' and CRC 1456 ``Mathematics of Experiment'' project A04. S.H.\ is funded by the German National Academy of Sciences Leopoldina under grant number LPDS 2024-11.

The authors thank C.\ Gottschlich for his help with the analysis of the fingerprint data.

\bibliographystyle{apalike2}
\bibliography{references}

\appendix
\newpage

\section{Proofs} \label{appendix:proofs}

\subsection{Hadamard Directional Derivatives}

In this subsection, we calculate the Hadamard directional derivatives of $\OT_c^\pm$ and $B_c^w$. We do this by rewriting them in terms of linear programs (LPs) and then employ the calculus developed for such LPs, see \cite{Gal1997}. For the definition and useful properties of directional Hadamard differentiability we refer to \cite{Roemisch2006,Shapiro1990}. Our proof strategy generalizes \cite{Sommerfeld2018} from probability measures to signed measures as well as from the two-sample case to general factorial designs.

We view a matrix $x \in \R^{N \times N}$ also as a vector $\vec{x} \in \R^{N^2}$ obtained by row-wise concatenation (and vice versa). Via this matrix-vector correspondence, define the linear row and column sum operators $\rowsum : \R^{N^2} \to \R^N$ and $\colsum : \R^{N^2} \to \R^N$, respectively. Let $\vec{c} \in \R^{N^2}$ be a cost vector (corresponding to a cost matrix $c \in \R^{N \times N}$). Then, for $\mu,\, \nu \in\Delta^{(E)}_N$ the OT problem can be written as
\begin{align} \label{eq:ot_lp}
  \OT_c(\mu,\nu) = &\min_{\vec{\pi} \in\R^{N^2}} \inner{\vec{c}}{\vec{\pi}} \\
  &\text{s.t. } A\vec{\pi}=b,\quad \vec{\pi}\geq 0\,, \notag
\end{align}
where $b \defeq [ \mu^\transp,\, \nu^\transp]^\transp$ and $A\in\R^{2N\times N^2}$ is defined such that $A\vec{\pi}= [\rowsum(\vec{\pi})^\transp,\, \colsum(\vec{\pi})^\transp]^\transp$.

Similarly, for $\mu^1, \ldots, \mu^\K \in \Delta_N$ and $w \in \Delta_{\K}$ the OT barycenter problem is equal to
\begin{align} \label{eq:bary_lp}
  B_c^w(\mu^1, \ldots, \mu^{\K}) = &\min_{\vec{\pi} \in\R^{KN^2}} \inner{\vec{c}_{\K}}{\vec{\pi}}  \\
  &\text{s.t. } B\vec{\pi}=d,\quad \vec{\pi}\geq 0\,, \notag
\end{align}
where $\vec{c}_{\K} \defeq [w_1 \vec{c}^{\,\transp}, \ldots, w_{\K} \vec{c}^{\,\transp}]^\transp$, $d \defeq [\mu^{1,\transp}, \ldots, \mu^{\K, \transp}, \bbzero_{(\K -1)N}^\transp]^\transp$ and $B \in \R^{\K N^2 \times \K N + (\K-1)N}$ is for $\vec{\pi} = [ \vec{\pi}_1^{\,\transp}, \ldots, \vec{\pi}_{\K}^{\,\transp}]^\transp$ defined by
\begin{equation*}
 B \vec{\pi} = [ \rowsum(\vec{\pi}_1)^\transp, \ldots, \rowsum(\vec{\pi}_{\K})^\transp,\, \colsum(\vec{\pi}_1 - \vec{\pi}_2)^\transp, \ldots, \colsum(\vec{\pi}_{\K - 1} - \vec{\pi}_{\K})^\transp ]^\transp\,.
\end{equation*}

\begin{theorem} \label{thm:had_deriv_otext}
  The mapping $\mu \mapsto \OText(\mu, \bbzero)$ is directionally Hadamard differentiable at all $\mu \in \Delta_N^\pm$ tangentially to $\Delta_N^\pm$ with derivative
  \begin{equation*}
    \Delta_N^\pm \ni h \mapsto \max_{(u, v) \in \Phi^*(\mu)} \inner{u}{ U_\mu(h) } + \inner{v}{ V_\mu(h) }\,.
  \end{equation*}
\end{theorem}
\begin{proof}
  Recall that for $\mu\in\Delta_N^\pm$ it holds that $\OText(\mu,\bbzero) = \OT_c(\mu^+, \mu^-)$. Hence, we can write 
  \begin{equation*}
    \OText(\mu,\bbzero) = (\OT_c \circ J)(\mu) \quad \text{with }J(\mu) \defeq [\max(\mu,0), \max(-\mu,0)]\,.
  \end{equation*}
  We show that both $\OT_c$ and $J$ are Hadamard directionally differentiable and then conclude with the chain rule \cite[Proposition~3.6]{Shapiro1990}.

  Since $\OT_c$ is an LP, see \eqref{eq:ot_lp}, \cite[Theorem~3.1]{Gal1997} yields that the Gateaux derivative of $\OT_c$ at $(\nu_1, \nu_2) \in \Delta_N^{(E)} \times \Delta_N^{(E)}$ is given by
  \begin{equation} \label{eq:ot_gat_deriv}
    (h^1,h^2)\mapsto \max_{(u,v)\in\Phi^*(\nu^1,\nu^2)} \inner{u}{h^1} + \inner{v}{h^2}\,,
  \end{equation}
  where $\Phi^*(\nu^1, \nu^2)$ is the set of optimal dual solutions of \eqref{eq:ot_lp}, i.e., it contains all $x = (u, v) \in \R^{N + N}$ with $\inner{x}{b} = \OT_c(\nu_1,\nu_2)$ and $A^\transp x \leq c$. This can be reformulated as $\inner{u}{\nu_1} + \inner{v}{\nu_2} = \OT_c(\nu_1, \nu_2)$ and $u \oplus v \leq c$. By \cite[Theorem~3.1]{Gisbert2019}, it follows that $\OT_c$ is locally Lipschitz continuous and therefore \cite[Proposition~3.5]{Shapiro1990} applies, i.e., $\OT_c$ is also directionally Hadamard differentiable with derivative  \eqref{eq:ot_gat_deriv}. Note that this is almost exactly the same derivative that was calculated in \cite[Theorem~4]{Sommerfeld2018}, but we also allow for measures that are not probability measures.

  For the directional Hadamard differentiability of $J$, note that the function $x \mapsto \max(x, 0)$ is directionally Hadamard differentiable at all $x \in \R$ with derivative
  \begin{equation*}
    h \mapsto h[ \I(x > 0) + \I(x = 0, h > 0)]\,.
  \end{equation*}
  The chain rule yields that $J$ is also directionally Hadamard differentiable at $\mu \in \Delta_N^{\pm}$ with derivative
  \begin{equation*}
    h \mapsto
  \begin{bmatrix}
    h(\I(\mu>0)+\I(\mu=0,h>0))\\
    -h(\I(\mu<0)+\I(\mu=0,h< 0))
    \end{bmatrix} = \begin{bmatrix}
      U_\mu(h) \\ V_\mu(h)
    \end{bmatrix}\,. \qedhere
  \end{equation*}
\end{proof}

\begin{theorem} \label{thm:had_deriv_bary}
  The OT barycenter functional $B_c^w$ is directionally Hadamard differentiable at all $\mu = (\mu^1, \ldots, \mu^{\K}) \in (\Delta_N)^{\K}$ tangentially to $(\Delta_N)^{\K}$ with derivative
  \begin{equation*}
   \mathcal{T}(\mu) \ni (h^1, \ldots, h^{\K}) \mapsto \maxsub{(u^1, \ldots, u^\K)}{\in \Psi^*(\mu^1, \ldots, \mu^\K)} \sum_{k=1}^\K \inner{u^k}{h^k}\,,
  \end{equation*}
  where $\mathcal{T}(\mu) \defeq \mathrm{Cl} \{ \frac{\tau - \mu}{t} : \tau \in (\Delta_N)^K,\, t > 0\}$ and $\mathrm{Cl}$ denotes the topological closure in $\R^{NK}$.
\end{theorem}
\begin{proof}
  As $B_c^w$ can be written as the LP \eqref{eq:bary_lp}, the above derivative follows from the same arguments used for $\OT_c$ to arrive at \eqref{eq:ot_gat_deriv}. Note that as we only perturb the first $\K N$ entries of the constraint vector $d$, the maximum can be reduced to the first $\K N$ dual optimal solutions which is exactly the set $\Psi^*(\mu^1, \ldots, \mu^{\K})$.
\end{proof}

\subsection{Distributional Limits}

We shortly recall the setting under which we will prove the distributional limits for the two statistics $\Tf(\hmu_n)$ and $\Tb(\hmu_n)$.
We are given $\mu^1,\ldots,\mu^\K \in \Delta_N^1$, $\K \geq 2$, with independent samples $X_1^k, \ldots, X^k_{n_k} \sim \mu^k$, $k \in \idn{\K}$, and sample sizes $n = (n_1,...,n_\K) \in \N^{\K}$. Denote with $\hmu^k_{n_k}$ the empirical probability vector based on the samples. Furthermore, write $\mu \defeq [\mu^1, \ldots, \mu^\K]^\transp$ and $\hmu_n \defeq [\hmu_{n_1}^1,\ldots,\hmu_{n_\K}^\K]^\transp$. For the given sample sizes $n = (n_1, \ldots, n_\K)$, define
\begin{equation*}
\rho_n \defeq \left( \prod_{k=1}^\K n_k\right)/\left(\sum_{k=1}^\K\prod_{j=1,j\neq k}^\K n_j\right),
\end{equation*}
and assume that for $n_\star \to \infty$ it holds
\begin{equation*}
  \rho_n/n_k \to \delta_k \in [0,1] \qquad \text{for all } k \in \idn{\K}.
\end{equation*}

Denoting for $\nu \in \Delta_N$ the multinomial covariance matrix
\begin{equation*}
  \Sigma(\nu) \defeq
  \begin{bmatrix}
   \nu_1 (1 - \nu_1) & -\nu_1 \nu_2 & \cdots & -\nu_1 \nu_N\\
  -\nu_2 \nu_1 &  \ddots & \ddots&  \vdots \\
  \vdots&\ddots&\ddots&-\nu_{N-1}\nu_N\\
  -\nu_N \nu_1 &  \cdots &  -\nu_{N} \nu_{N-1} & \nu_N (1-\nu_N)
  \end{bmatrix} \in \R^{N \times N},
\end{equation*}
the basis for our limit distributions is the following limit law:
\begin{lemma} \label{lemma:limit_dist_mu}
It holds for $n_\star\rightarrow\infty$ that
  \begin{align*}
    \sqrt{\rho_n} ( \hmu_n - \mu ) \dconv G =
      [G^1, \ldots, G^\K ]^\transp,
  \end{align*}
  where $G^k \sim \calN_N(0, \delta_k \Sigma(\mu^k))$, $k \in \idn{\K}$, are independent.
\end{lemma}
\begin{proof}
  By \cite[Theorem~14.6]{Wassermann2011} and Slutzky's theorem, it follows that $\sqrt{\rho_n} [ \hmu_{n_k}^k - \mu^k ] \dconv G^k$ as $n_\star\to\infty$ for all $k \in \idn{\K}$. Due to the independence of $\hmu_{n_1}^1, \ldots, \hmu_{n_{\K}^{\K}}$, it follows from \cite[Example~1.4.6]{Vaart2023} that we also have joint convergence.
\end{proof}

\begin{proof}[Proof of \autoref{thm:limit_dist_fac}]
  An application of the continuous mapping theorem to \autoref{lemma:limit_dist_mu} yields
  \begin{equation} \label{eq:limit_dist_LR}
    \sqrt{\rho_n} [ L\hmu_n - L\mu ] \dconv LG\,.
  \end{equation}
  Now, define $Q:(\Delta_N^\pm)^M \to \R^{M}$ by
\begin{equation*}
  (\mu^1, \ldots, \mu^M) \mapsto \bigl[ \OText(\mu^m, \bbzero) \bigr]_{m=1}^M\,.
\end{equation*}
Note that $Q$ is Hadamard directionally differentiable and the derivative is obtained componentwise by \autoref{thm:had_deriv_otext}. Hence, an application of the delta method \citep{Roemisch2006} to \eqref{eq:limit_dist_LR} yields the desired result.
\end{proof}

\subsubsection{Distribution of Distances}

From \autoref{subsec:dod}, recall the ``distance'' $d : \calX \times \calX \to \{d_1, \ldots, d_R\}$ which takes $R$ distinct values, and the distributions of distances $\nu$ of $\mu$ with empirical version $\hnu_n$. Furthermore, for $\tau \in \Delta_N$ define the covariance matrix $\Sigma(\tau, d) \in \R^{R \times R}$ componentwise by 
\begin{equation*}
  \Sigma(\tau, d)_{r,s} = 4 \left[ \sum_{z, q, p = 1}^N \I(d(x_z, x_q) = d_r,\, d(x_z, x_p) = d_s) \tau_z \tau_p \tau_q - \nu_r \nu_s \right]\,.
\end{equation*}

\begin{lemma} \label{lemma:dod_dist}
  It holds for $n_\star \to \infty$ that 
  \begin{equation*}
    \sqrt{\rho_n} [ \hnu_n - \nu ] \dconv [H^1, \ldots,  H^{\K}]^\transp\,,
  \end{equation*}
  where $H^k \sim \normaldist(0, \delta_k \Sigma(\mu^k, d))$, $k\in\idn{\K}$, are independent.
\end{lemma}
\begin{proof}
  We show that $\hnu^k_{n_k}$ can be written as a $U$-statistic and derive the asymptotic distribution. Indeed, we have $\EV{\hnu^k_{n_k}} = \nu^k$ and 
  \begin{equation*}
    \hnu^k_{n_k, r} = \binom{n}{2}^{-1} \sum_{1 \leq p < q \leq n_k} \psi_r(X^k_p, X^k_q) \,, \qquad r \in \idn{R},\, k \in \idn{\K}\,,
  \end{equation*}
  with kernel (of order $2$) $\psi_r(x, y) \defeq \I(d(x, y) = d_r)$ for $x, y\in \calX$. Hence, \citep[Theorem~2 in Chapter~3]{lee1990ustat} and Slutzky's theorem imply that $\sqrt{\rho_n} [\hnu^k_{n_k} - \nu^k]$, as $n_\star \to \infty$, converges in distribution to zero-mean Gaussian random variable with covariance matrix $\delta_k S^k$, where $S^k \in \R^{R \times R}$ is given componentwise  by 
  \begin{align*}
    S^k_{r, s} &\defeq 2 \cdot 2 \Cov[\psi_r(X_1, X_2), \psi_s(X_2, X_3)] \\
    &\,= 4 [ \Prob{d(X_1, X_2) = d_r,\, d(X_2, X_3) = d_s} \\
    &\qquad- \Prob{d(X_1, X_2) = d_r} \Prob{d(X_2, X_3) = d_s} ] = \Sigma(\mu^k, d)_{r, s}\,,
  \end{align*} 
  i.e., $S^k = \Sigma(\mu^k, d)$. Finally, the joint convergence follows by noting that $\hnu^1_{n_1}, \ldots, \hnu^{\K}_{n_{\K}}$ are independent.
\end{proof}

\subsection{Hypothesis Testing}

\begin{proof}[Proof of \autoref{thm:fac_test_power}]
  As $\hmu^k_{n_k, i} = \frac{1}{n_k} \sum_{r=1}^{n_k} \I(X_r^k = i)$, $i \in \idn{N}$, it follows by the strong law of large numbers that $\hmu^k_{n_k} \asconv \mu^k$. Notably, this implies that $\Sigma(\hmu^k_{n_k}) \asconv \Sigma(\mu^k)$ and thus for $\hG^k_{n_k} \sim \normaldist(0, \Sigma(\hmu^k_{n_k}))$ that $\hat{G}^k_n \dconv G^k$. By independence and the continuous mapping theorem, it follows that
  \begin{equation*}
    \frac{1}{s} \sum_{m=1}^M \OText([L\hat{G}_n]_m, \bbzero) \dconv \frac{1}{s} \sum_{m=1}^M  \OText([LG]_m, \bbzero) \,.
  \end{equation*}
 As $\hq_{1-\alpha,n}$ and $q_{1-\alpha}$ are the $(1-\alpha)$-quantiles of the l.h.s.\ and r.h.s., respectively, we conclude that $\hq_{1-\alpha,n}$ is bounded as $n_\star \to \infty$. Call $H$ the limiting variable of $\Tf(\hmu_n) - \Tf(\mu)$ given in \autoref{thm:limit_dist_test_fac}. Note that for any fixed choice of $(u_m^*,v_m^*)\in\Phi^*([L\mu]_m)$, $m \in \idn{M}$, it holds that
  \begin{equation*}
   H \geq \tilde{H} \defeq \frac{1}{s} \sum_{m=1}^M \inner{u_m^*}{U_{L\mu}(LG)_m} + \inner{v_m^*}{V_{L\mu}(LG)_m} \,.
  \end{equation*}
  The Cauchy Schwarz inequality and the definitions of $U_{L\mu}(LG)$, $V_{L\mu}(LG)$ yield that
  \begin{align*}
   \abs{\tilde{H}} &\leq \frac{1}{s} \sum_{m=1}^M \norm{u^*_m}_2 \norm{U_{L\mu}(LG)_m}_2 + \norm{v^*_m}_2 \norm{V_{L\mu}(LG)_m}_2 \\
   &\leq \frac{1}{s} \sum_{m=1}^M (\norm{u^*_m}_2 + \norm{v^*_m}) \norm{[LG]_m}_2\,.
  \end{align*}
  As $G^1, \ldots, G^\K$ are Gaussian, it follows that
  \begin{equation*}
   1 \geq \Pr{H \geq x} \geq \Pr{\tilde{H} \geq x} \geq \Pr{ \abs{\tilde{H}} \leq \abs{x}} \to 1 \qquad \text{as } x \to -\infty\,,
  \end{equation*}
  and therefore $\Pr{H \geq x} \to 1$ as $x \to -\infty$. Moreover, due to \autoref{thm:limit_dist_test_fac} for all $\eps > 0$ there exists $N \in \N$ such that $\Pr{\Tf(\hmu_n) - \Tf(\mu) \geq x} \geq \Pr{H \geq x} - \eps$ for all $n \geq N$. Hence, we obtain that
  \begin{align*}
    \Pr{\hphif(\hmu_n) = 1} &= \Pr{ \Tf(\hmu_n) \geq \hq_{1-\alpha,n}} \\
     &= \Pr{\Tf(\hmu_n) - \Tf(\mu)  \geq \hq_{1-\alpha,n} - \Tf(\mu)} \\
    &\geq \Pr{ H \geq \hq_{1-\alpha,n} - \Tf(\mu) } - \eps \to 1 - \eps \qquad \asntoinfty\,,
  \end{align*}
  since $\hq_{1-\alpha,n}$ is bounded and under the alternative $\Tf(\mu) \to \infty$. Letting $\eps \to 0$ yields the claim.
 \end{proof}

\section{Barycenter Method} \label{sec:bary}

This section is concerned with statistical properties of the barycenter method and a comparison with FDOTT. 

\subsection{Limit Distribution of the Barycenter Test Statistic}
In this subsection, we derive the limit distribution of the test statistic $\Tb(\hmu_n)$ for the barycenter method as defined in \eqref{eq:test_stat_bary}. 

\begin{theorem}\label{thm:limit_dist_bary}
Let $v^0 \defeq \bbzero_N \eqdef v^{\K}$ and denote $\Psi^*(\mu^1, \ldots, \mu^{\K})$ the convex set  containing all $(u^1, \ldots, u^{\K}) \in \R^{\K N}$ such that there exist $(v^1, \ldots, v^{\K-1}) \in \R^{(\K-1)N}$ with
\begin{equation} \label{eq:Psi}
  \sum_{k=1}^{\K} \inner{\mu^k}{u^k} = B_c^w(\mu^1, \ldots, \mu^{\K})\,, \qquad u^k \oplus (v^k - v^{k-1}) \leq w_k c,\;\; k \in \idn{\K}\,,
\end{equation}
where $w \in \Delta_{\K}$ is a fixed positive weight vector as in \eqref{eq:ot_bary}. Then, under the setting in \autoref{subsec:approach}, as $n_\star\rightarrow\infty$,
\begin{equation*}
 \Tb(\hmu_n) - \Tb(\mu) \dconv\max_{u\in\Psi^*(\mu^1,...,\mu^{\K})}\inner{u}{G}\,.
\end{equation*}
\end{theorem}
\begin{proof}
  Follows directly from an application of the delta method to \autoref{lemma:limit_dist_mu} in conjunction with \autoref{thm:had_deriv_bary}.
\end{proof}

\begin{remark} \label{rem:simplified_Psi}
Assume that the null hypothesis $\nullb$ in \eqref{eq:null_bary} holds, i.e., $\mu^1=...=\mu^{\K}$, then it is possible to simplify $\Psi^*(\mu^1,...,\mu^{\K})$. Indeed, in this case we have $B_c^w(\mu^1,...,\mu^{\K})=0$, so the constraint on the r.h.s.\ of \eqref{eq:Psi} simplifies to
\begin{equation*}
  \sum_{k=1}^{\K} \inner*{\mu^1}{u^{k}} = \sum_{i=1}^N\left(\mu^1_i\cdot\sum_{k=1}^{\K} u^{k}_i\right)=0\,.
\end{equation*}
Since $c_{ii}=0$, $i \in \idn{N}$, it follows that
\[\sum_{k=1}^{\K} u^k \leq \sum_{k=1}^{\K}(v^{k-1} - v^{k})= v^{0}-v^{\K} =\bbzero_N \,.\]
Hence, it holds $\sum_{k=1}^{\K} u^{k}_i = 0$ for all $i \in \supp \mu^1 \defeq \{ j \in \idn{N} : \mu_j^1 > 0\}$. In particular, if $\mu^1$ has full support, $\Psi^*(\mu^1,\ldots,\mu^{\K})$ does not depend on $\mu^1$ and is equal to
\begin{equation*}
  \Psi^*=\left\{ (u^1, \ldots, u^{\K}) \in \R^{\K N } : \begin{aligned}
    &\exists (v^1, \ldots, v^{\K-1}) \in \R^{(\K-1)N}, \text{ s.t.}  \\
    &\sum_{k=1}^{\K} u^k = \bbzero_N,\; u^k \oplus (v^k - v^{k-1}) \leq w_k c,\; k \in \idn{\K}\,.
  \end{aligned}\right\}\,.
\end{equation*}
For simplicity, we always assume in the following that the probability measures have full support. This can be achieved by discarding points that do not lie in the joint support of the probability measures, i.e., reducing $\idn{N}$ to $\bigcup_{k \in \idn{\K}} \supp \mu^k$. Else, the following results still hold where $\Psi^*$ is substituted by $\Psi^*(\mu^1,\ldots,\mu^{\K})$.
\end{remark}

With the simplifications discussed above, the limit law in \autoref{thm:limit_dist_bary} under the null hypothesis $\nullb$ reduces to the following.

\begin{corollary} \label{cor:limit_dist_null_bary}
  Suppose the null hypothesis $\nullb$ in \eqref{eq:null_bary} (with $\mu^1$ having full support) holds. Then, under the setting of \autoref{thm:limit_dist_bary}, as $n_\star \to \infty$,
\begin{equation*}
  \Tb(\hmu_n) \dconv\max_{u\in\Psi^*} \inner{u}{G}\,.
\end{equation*}
\end{corollary}

\begin{remark}[Normality] \label{rem:bary_normality}
If the set $\Psi^*(\mu^1,...,\mu^{\K})$ is a singleton, then the limiting distribution in \autoref{thm:limit_dist_bary} is Gaussian. Analogous as in \autoref{rem:fac_normality}, under the null hypothesis $\nullb$ the limiting distribution will never be Gaussian. Under the alternative, a sufficient condition is given in \cite[Condition~(A1)]{Kravtsova2025}.
\end{remark}

Finally, note that all sampling schemes given in \autoref{subsubsec:FDOTT_action} for FDOTT remain valid for the barycenter method.

\subsection{Comparison of FDOTT and the Barycenter Method} \label{subsec:comp}

The FDOTT in the one-way layout (\autoref{ex:one_way_1}) and the barycenter method both can test for the null $\nullb$ in \eqref{eq:null_bary}. In the following, we compare the two approaches. 

\subsubsection{Local Power} \label{subsec:local_power}

It follows immediately from \autoref{thm:limit_dist_null_test_fac} and \autoref{cor:limit_dist_null_bary} that under fixed alternatives, both FDOTT and the barycenter method have asymptotic power $1$. For a more detailed comparison, we will now analyze the asymptotic power of these tests at \emph{local} alternatives. Assume that $\mu^1, \ldots, \mu^{\K} \in \Delta_N$ satisfy the null $\nullf$ (or the special case $\nullb$), i.e., it holds that $L\mu = \bbzero$. We perturb these by $\nu^1,\ldots,\nu^{\K}\in\Delta_N$: Set
\begin{equation} \label{eq:loc_alt}
  \mu^k_{n_k} \defeq \frac{1}{\sqrt{n_k}}\nu^k+(1-\frac{1}{\sqrt{n_k}})\mu^{k}\,, \quad k \in \idn{\K}\,,
\end{equation}
and suppose that we observe i.i.d.\ samples $X_{1}^k, \ldots, X_{n_k}^k \sim \mu^k_{n_k}$, $k \in \idn{\K}$. As before, let $\hmu_{n_k}^k$ be the empirical probability measures based on these. Now, the null hypothesis is approximated by a sequence of measures $\mu^1_{n_1}, \ldots, \mu^{\K}_{n_{\K}}$, $n_\star\to\infty$, see \cite[Chapter~14]{Vaart1998}. A particularly simple case occurs if $n_1 = \ldots = n_{\K}$ and $\nu \defeq [\nu^1, \ldots, \nu^{\K}]^\transp$ does not satisfy $L\nu = \bbzero$, then it follows that $\mu_n = \frac{1}{\sqrt{n_1}} \nu + (1-\frac{1}{\sqrt{n_1}}) \mu$ does not satisfy the null hypothesis $\nullf$ for any $n_1 \in \N$ either.

\begin{theorem}\label{thm:local_limit}
  Denote the scaled differences $\eta \defeq [\sqrt{\delta_1} (\nu^1 - \mu^{1}), \ldots, \sqrt{\delta_{\K}} (\nu^{\K} - \mu^{\K}) ]^\transp$.
  \begin{itemize}
    \item[a)] (Local asymptotics for FDOTT). Suppose that we are in the setting of local alternatives \eqref{eq:loc_alt} for $\nullf$. Then, as $n_\star \to \infty$,
\begin{equation*}
  \Tf(\hmu_n) \dconv \frac{1}{s} \sum_{m=1}^M \OText([L(G + \eta)]_m, \bbzero)\,.
\end{equation*}
    \item[b)] (Local asymptotics for the barycenter method). Suppose that we are in the setting of local alternatives \eqref{eq:loc_alt} for $\nullb$ (with $\mu^1$ having full support). Then, as $n_\star \to \infty$,
  \begin{equation*}
    \Tb(\hmu_n) \dconv \max_{u \in \Psi^*} \inner{u}{G + \eta}\,,
  \end{equation*}
  where $\Psi^*$ is defined in \autoref{rem:simplified_Psi}.
  \end{itemize}
\end{theorem}
\begin{proof}
  For any $k \in \idn{\K}$, it holds that
  \begin{equation*}
    \sqrt{\rho_n}(\hmu^k_{n_k} - \mu^{k}) = \sqrt{\rho_n}(\hmu_{n_k}^k-\mu_{n_k}^k) + \sqrt{\rho_n}(\mu^k_{n_k}-\mu^{k}) = \sqrt{\rho_n}(\hmu_{n_k}^k-\mu_{n_k}^k) + \sqrt{\frac{\rho_n}{n_k}} [\nu^k-\mu^{k} ]\,.
  \end{equation*}
  First, note that $\sqrt{\frac{\rho_n}{n_k}} [\nu^k-\mu^{k} ] \dconv \sqrt{\delta_k} [\nu^k - \mu^{k}]$ as $n_\star \to \infty$. Furthermore, we have $n_k \hmu^k_{n_k} \sim \Mult(n_k,\mu_{n_k}^k)$ and thus,
  \[\EV{\hmu^k_{n_k}} = \mu^k_{n_k}, \quad \Cov(\hmu^k_{n_k})=\Sigma(\mu^k_{n_k})\xrightarrow{n_k\to\infty} \Sigma(\mu^{k}).\]
  Thus, by the multivariate central limit theorem \cite[Proposition~2.27]{Vaart1998},
  \begin{equation*}
    \sqrt{\rho_n}(\hmu^k_{n_k}-\mu_{n_k}^k)\dconv G^k\,,
  \end{equation*}
  and therefore by Slutzky's theorem
  \begin{equation*}
    \sqrt{\rho_n}(\hmu^k_{n_k}-\mu^{k}) \dconv \sqrt{\delta_k} [\nu^k - \mu^{k}] + G^k = [\eta + G]_k \,.
  \end{equation*}
  Due to the independence of the different samples, this implies joint convergence over all $k \in \idn{\K}$, i.e.,
  \begin{equation*}
    \sqrt{\rho_n}( \hmu_n - \mu) \dconv \eta + G\,,
  \end{equation*}
  Analogous as in the proof of \autoref{thm:limit_dist_fac}, it follows that
  \begin{equation*}
    \Tf(\hmu_n) - \Tf(\mu) \dconv \frac{1}{s} \sum_{m=1}^M \OText([L(G + \eta)]_m, \bbzero)\,.
  \end{equation*}
  We conclude by noting that under the null hypothesis $\nullf$ it holds $\Tf(\mu) = 0$.

  The proof for $\Tb(\hmu_n)$ follows analogously.
\end{proof}

Call $H(G)$ the limit distribution of $\Tf(\hmu_n)$ under the null given in \autoref{thm:limit_dist_null_test_fac}, then the local limit in \autoref{thm:local_limit} can be expressed as $H(G + \eta)$. Hence, the local power for $n_\star \to \infty$ is given by $\Pr{H(G + \eta) \geq q_{1-\alpha}}$. Recall that $q_{1-\alpha}$ is the $(1-\alpha)$-quantile of $H(G)$. Thus, the asymptotic local power depends on how the quantiles of $H(G)$ are affected by the shift $G + \eta$. Due to the complicated structure of $H(G)$, a direct comparison between the limits in a) and b) is not obvious (even in the special case of the one-way layout in \autoref{ex:one_way_1}). Therefore, we compare the local power for $\Tf(\hmu_n)$ and $\Tb(\hmu_n)$ in specific situations via simulations, see \autoref{appendix:sim}.

\subsubsection{Practical Aspects} \label{subsubsec:practical_diff}

We discuss further practical differences of FDOTT and the barycenter method.

\textbf{Range of applicability.} The barycenter method can only be employed in the one-way layout, i.e., for testing $\mu^1 = \ldots = \mu^{\K}$. In contrast, FDOTT can also deal with higher-way layouts for testing any linear relationship $L \mu = \bbzero$ in \eqref{eq:null_fac}. Furthermore, note that post-hoc testing strategies are immediately available for FDOTT, see \autoref{subsubsec:post_hoc} and in particular Tukey's HSD test (\autoref{ex:tukey_hsd_test}). Using the barycenter method, it is not obvious to us how to apply post-hoc strategies in a reasonable manner.

\textbf{Computation.} For practical purposes the efficient computation of the statistics $\Tf(\hmu_n)$ and $\Tb(\hmu_n)$ and the simulation of their quantiles is most important. To compute $\Tb(\hmu_n)$ for the barycenter method, and to simulate once from the limiting distribution, we have to solve one LP with $\K N+(\K-1)(N-1)+1$ constraints and solution vector length $\K N^2$, for FDOTT we need to solve $M$ ($=\K(\K-1)/2$ in the one-way layout) OT problems. For comparison, the resulting OT problem is an LP with $2N$ constraints and solution vector length $N^2$. In contrast to the barycenter method, however, we can employ OT solvers that are much faster than general LP solvers, due to their specialized nature \citep{Bertsekas1991}. Furthermore, these $M$ separate OT problems can also be computed in parallel. All in all, computation of FDOTT is generally faster and less memory intensive for large $N$ and $\K$. Note, that the sizes of all the mentioned LPs scale quadratically with $N$, so increasing the number of support points scales the run time correspondingly, see \cite{Schrieber2016} for a run time comparison of several exact solvers.

\subsubsection{Sandwich Inequality} \label{subsubsec:ineq_apad_bary}

For Euclidean spaces the relation given in \eqref{eq:idea} between the barycenter and the pairwise squared distances holds because these are linear inner product spaces, see e.g.\ \citet[Proposition~5.4]{Sturm2003}. However, it is known that the Euclidean $2$-Wasserstein space is not flat, but has postive curvature \citep[Theorem~7.3.2]{Ambrosio2008}. This indicates that the strong relation as given in \eqref{eq:idea} does not hold for the two statistics $\Tb(\hmu_n)$ and $\Tf(\hmu_n)$, in general. Nevertheless, we show that they dominate each other by a factor of $2$ at most.

\begin{theorem}[Sandwich inequality for pairwise comparisons] \label{thm:ineq_apad_bary}
  Suppose that we are in the one-way layout (\autoref{ex:one_way_1}), let $c$ be a metric cost matrix and $w = (1/\K, \ldots, 1/\K)^\transp$. Then, it holds for any $\mu^1, \ldots, \mu^{\K} \in \Delta_N$ that
\begin{equation*}
  \Tf(\mu) \leq \Tb(\mu) \leq 2\,\Tf(\mu)\,.
\end{equation*}
\end{theorem}

To prove \autoref{thm:ineq_apad_bary}, we first derive a more general result on metric spaces. Its proof is elementary and it seems to be folklore in the community (see e.g.\ \citealp{Sturm2003}), however, we did not find a complete proof and explicit statement, hence we give it here.

\begin{theorem} \label{thm:ineq_apad_bary_metric_space}
  Let $X$ and $X'$ be independent and identically distributed (Borel) random variables on a metric space $(\calX, d)$ such that $\EV{d^p(X,X')} < \infty$ for $p\in [1, \infty)$. Then, it follows that
  \begin{equation*}
    2^{-p} \EV{d^p(X,X')} \leq \inf_{a \in \calX} \EV{d^p(a, X')} \leq \EV{d^p(X,X')}\,.
  \end{equation*}
\end{theorem}
\begin{proof}
  For the first inequality, notice by triangle inequality combined with $(x+y)^p \leq 2^{p-1} (x^p + y^p)$ for $x,y\geq 0$, and since $X,\,X'$ have the same distribution that
  \begin{align*}
    \EV{d^p(X,X')} &\leq \EV{(d(X,a) + d(a,X'))^p} \\
    &\leq 2^{p-1}\EV{d^p(X,a) + d^p(a,X')} \leq  2^{p} \EV{d^p(a,X')}\,.
  \end{align*}
Dividing both sides of the above display by $2^p$ and taking the infimum over $a\in \calX$ implies the first inequality.

To show the second inequality denote by $P$ the distribution of $X$ and observe that
\begin{equation*}
  \inf_{a\in \calX} \EV{d^p(a,X')} \leq \EVV{X}{\EVV{X'}{d^p(X,X')}} = \EV{d^p(X,X')}\,,
\end{equation*}
where we used Fubini-Tonelli for the last equality.
\end{proof}

\begin{proof}[Proof of \autoref{thm:ineq_apad_bary}]
  Consider the case $(\calX, d) = (\Delta_N, \OT_c)$. Take $P = \frac{1}{\K} \sum_{k=1}^{\K} \delta_{\mu^{k}}$ and let $X,\,X' \sim P$ be i.i.d. Then, noting that $2\, \Tf(\mu) = \sqrt{\rho_n} \EV{d(X, X')}$ (recall \autoref{ex:one_way_1}) and $\Tb(\mu) = \sqrt{\rho_n} \inf_{a \in \calX} \EV{d(a, X')}$, the assertion directly follows from \autoref{thm:ineq_apad_bary_metric_space}.
\end{proof}

\newpage

\section{Applications} \label{appendix:cell}

\subsection{Microtubules in Cytoskeletons}

We give more details on our analysis of the influence of vimentin intermediate filaments and actin filaments on the appearance of microtubule structures in the cytoskeleton as discussed in \autoref{subsec:appl}.

Our analysis is based on microtubule histograms which display the location distribution of microtubules in the cytoskeleton extracted from images in \cite{Blob2024}. As in \cite{Blob2024}, these images were aligned with respect to the cell center and microtubules were extracted. In the data that we analyzed, the microtubules with length less than \qty{1.56}{\micro\meter} were discarded to reduce the effects of extraction artifacts. For each point of a microtubule its local curvature was calculated by fitting an osculating circle, and the curvature of the whole microtubule was set to the mean of the local curvatures. Furthermore, we assigned two further values to each microtubule: The distance (of the mean of the microtubule points) to the cell center and the angular orientation relative to the cell center. The latter was obtained by fitting a line through the microtubule. This procedure is illustrated in \autoref{fig:microtubules_hist_prod}. These values are calculated for all microtubules in all cells.

As microtubules in one cell are strongly locally dependent, they may not be considered i.i.d. To reduce this local dependency, we subsample from each cell. To this end, each cell is divided by distance and angle from the cell center. Then, we randomly pick one microtubule from each region. This procedure is illustrated in \autoref{fig:microtubules_hist_prod}. We chose to consider 10 divisions for distance and 15 for the angle. That means that per cell, only 150 microtubules may be picked. After subsampling, we aggregate all microtubules from a group and then create the histogram, see \autoref{tab:cells_samplesizes} for the sample sizes. For the number of bins, we chose 10 for the distance, 6 for the curvature and 8 for the orientation. We compute each histogram on a regular grid of size $10 \times 6 \times 8$ and used the Euclidean distance to create the cost matrix $c$.

Recall from \autoref{subsec:appl} that we have a three-way layout with, say, factors C (for cell line), V (for vimentin) and A (for actin) with $\K_1 = 2$, $\K_2 = 2$ and $\K_3 = 2$ levels, respectively. Denote the underlying histograms for each factor combination by $\mu^{ijl}$, $i\in\idn{\K_1}$, $j\in\idn{\K_2}$, $l\in\idn{\K_3}$. In this three-way layout we test all interaction and main effects, see \autoref{tab:cells_anova} for the results. The null hypothesis for the three-way interaction {C~$\times$~V~$\times$~A} reads
\[\nullf: \mu^{ijl} - \mu^{i\bullet\bullet}-\mu^{\bullet j \bullet}-\mu^{\bullet\bullet l}+\mu^{ij\bullet}+\mu^{i\bullet l}+\mu^{\bullet jl}-\mu^{\bullet\bullet\bullet} = 0\]
for all $i\in\idn{\K_1}$, $j\in\idn{\K_2}$, $l\in\idn{\K_3}$. Here, we use the same notation for the mean vectors as in \eqref{eq:mu_mean}. The remaining hypotheses of interaction and main effects can be reduced to the two-way layout (\autoref{ex:two_way_1}) by averaging. For instance, considering $\mu^{ij\bullet}$, $i\in\idn{\K_1}$, $j\in\idn{\K_2}$ yields the interaction C $\times$ V.

\begin{table}
\begin{center}
\begin{tabular}{c|c|c|c}
 &  & None &  LatA + Noc \\
\hline
MEF WT & No. of cells & 82 & 61 \\
 & No. of microtubules & 8769 & 6348 \\ \hline
MEF VimKO & No. of cells &  65 & 57 \\
 & No. of microtubules &  6765 & 5808  \\ \hline
 NIH3T3 & No. of cells & 102 & 26 \\
 & No. of microtubules & 10894 & 2727 \\ \hline
 NIH3T3 VimKO & No. of cells & 77 & 26 \\
 & No. of microtubules & 8164 & 2804
\end{tabular}
\end{center}
\caption{Number of cells and microtubules for all factor combinations.}
\label{tab:cells_samplesizes}
\end{table}

\begin{table}
\begin{center}
\begin{tabular}{c|c|c|c|c|c|c|c}
 interaction & C $\times$ V $\times$ A  & C $\times$ V & C $\times$ A & V $\times$ A & C & V & A \\\hline
$p$-value & 1.000 & 1.000 & 0.978 & 0.614 & 0.028 & 0.856 & 0.125
\end{tabular}
\end{center}
\caption{Results of FDOTT for the interaction and main factor effects of the three-way layout of the microtubule histograms. Here, ``$\times$'' denotes an interaction effect between a combination of factors I (C for cell line), II (V for vimentin) and III (A for actin). No ``$\times$'' means main effect. The $p$-value is computed via the procedure from \autoref{subsec:appl}.}
\label{tab:cells_anova}
\end{table}

\begin{figure}
   \centering
   \begin{subfigure}[b]{0.45\textwidth}
       \includegraphics[width=\textwidth]{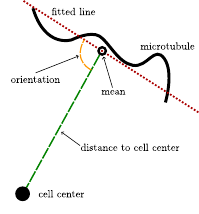}
   \end{subfigure}
   \hfill 
   \begin{subfigure}[b]{0.45\textwidth}
       \includegraphics[width=\textwidth]{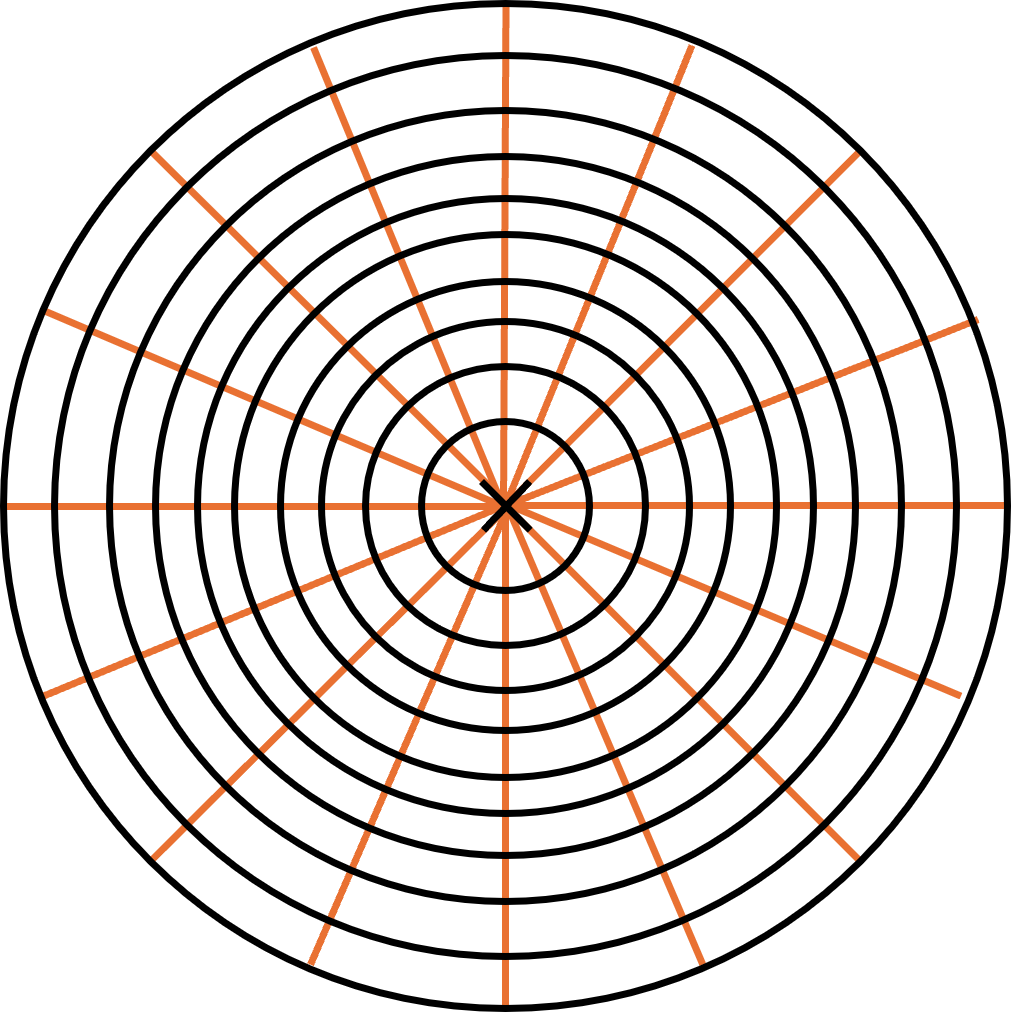}
   \end{subfigure}
   \caption{Sketches to illustrate the creation of microtubules histograms. The left sketch shows how the values distance and orientation are assigned to the microtubule. The right sketch shows how each cell is divided for the subsampling method.}
   \label{fig:microtubules_hist_prod}
\end{figure}

\newpage

\subsection{Real and Synthetic Fingerprints}

In \autoref{tab:fingerprints_samplesizes} we provide the sample sizes of the minutiae histograms for the discrimination of synthetic and real fingerprints in \autoref{sec:fingerprints}.

\begin{table}[b]
\begin{center}
\begin{tabular}{c|c|c|c|c}
 &  & Left Loop & Right Loop & Whorl \\
\hline
Real & No.\ of fingerprints &  36& 35 &  26\\
 & No.\ of minutiae &  1363 & 1263  & 1089  \\ \hline
Synthetic & No.\ of fingerprints & 42 & 36 & 22 \\
 & No.\ of minutiae & 1199 & 955 & 728 \\
\end{tabular}
\end{center}
\caption{Number of fingerprints and minutiae for each factor combination.}
\label{tab:fingerprints_samplesizes}
\end{table}

\newpage
\section{Simulations} \label{appendix:sim}

In this section, we conduct an extensive simulation study to investigate the convergence of the limit laws and to assess the finite sample validity, for both FDOTT and the barycenter test. Furthermore, we simulate the local power for both tests.

\subsection{Convergence Simulations}

We consider the setting of \autoref{thm:limit_dist_test_fac} in the in one-way layout (\autoref{ex:one_way_1}) and in the two-way layout (\autoref{ex:two_way_1}) as well as \autoref{thm:limit_dist_bary}. We generate measures on the regular $L\times L$ grid, i.e., $\{1,...,L\}^2$, creating the cost matrix $c$ by taking the Euclidean distance, and choose the hyperparameters $L \in \{3,5,10\}$,  $n \in \{ 50,500, 5000\}$ and $\K \in \{5,10\}$ for the one-way layout, and $(\K_1,\K_2) \in \{(2,2),(4,4)\}$ for the two-way layout. The sample sizes are taken to be equal $n=n_1=...=n_\K$. From the finite sample and limit distributions of the aforementioned limit laws we simulate $3000$ values and plot the corresponding kernel density estimator (with Gaussian kernel and bandwidth chosen according to Silverman's rule of thumb \cite{Silverman1986}), respectively.

As noted in \cite{Sommerfeld2018} for the two-sample case, we found that the maximization problems in the limiting distribution for \autoref{thm:limit_dist_test_fac} are numerically unstable when the null hypothesis does not hold. Hence, in this case, we use the reformulation as a non-linear optimization problem as proposed there.

\paragraph{One-Way Simulations.}

In the one-way layout, we analyze the convergence of the limit law for the FDOTT statistic given in \autoref{ex:one_way_1}. For this, we draw $\mu^1$ uniformly from $\Delta_{L^2}$ and set $\mu^1 \eqdef \mu^2 = \ldots = \mu^{\K}$. For the case that the null hypothesis does not hold, we just draw $\K$ measures independently and uniformly from $\Delta_{L^2}$.

Under the null hypothesis, we see in \autoref{fig:conv_H0_one_way_kde} that the convergence is quite fast. Also, with larger grid sizes $L$ and numbers of measures $\K$, larger sample sizes $n$ are needed to accurately estimate the limiting distribution.

When the null hypothesis does not hold, \autoref{fig:conv_H1_one_way_kde} shows that the convergence is much slower. In particular, there seems to be a negative bias that only slowly decreases with larger sample sizes. Again, increasing $\K$ and $L$ seems to decrease convergence speed.

\begin{figure}
   \centering
   \includegraphics[width=0.95\textwidth]{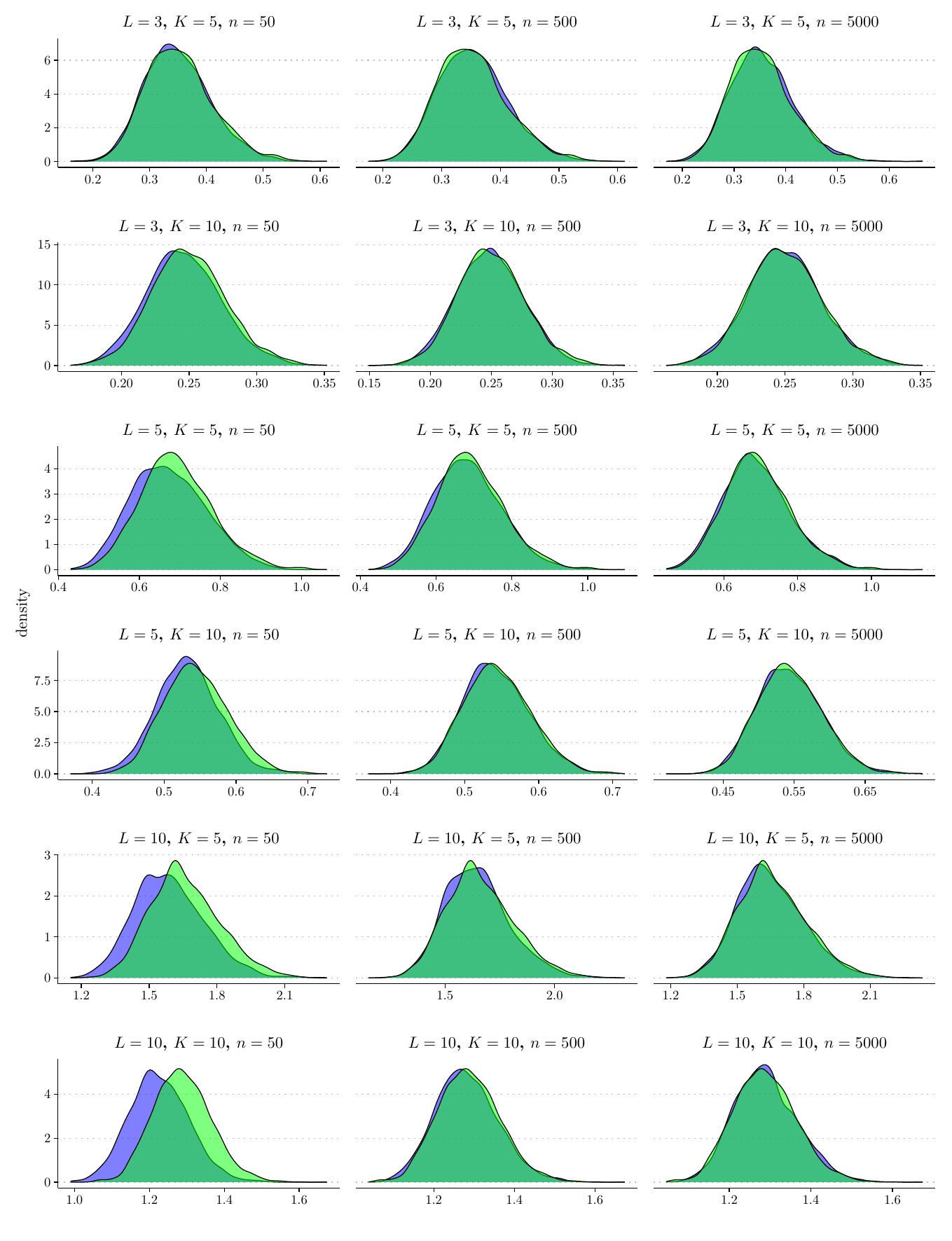}
   \caption{Comparison of the finite sample (blue) and the limiting density (green) in the one-way layout under the null hypothesis for different parameters.}
   \label{fig:conv_H0_one_way_kde}
\end{figure}

\begin{figure}
   \centering
   \includegraphics[width=0.95\textwidth]{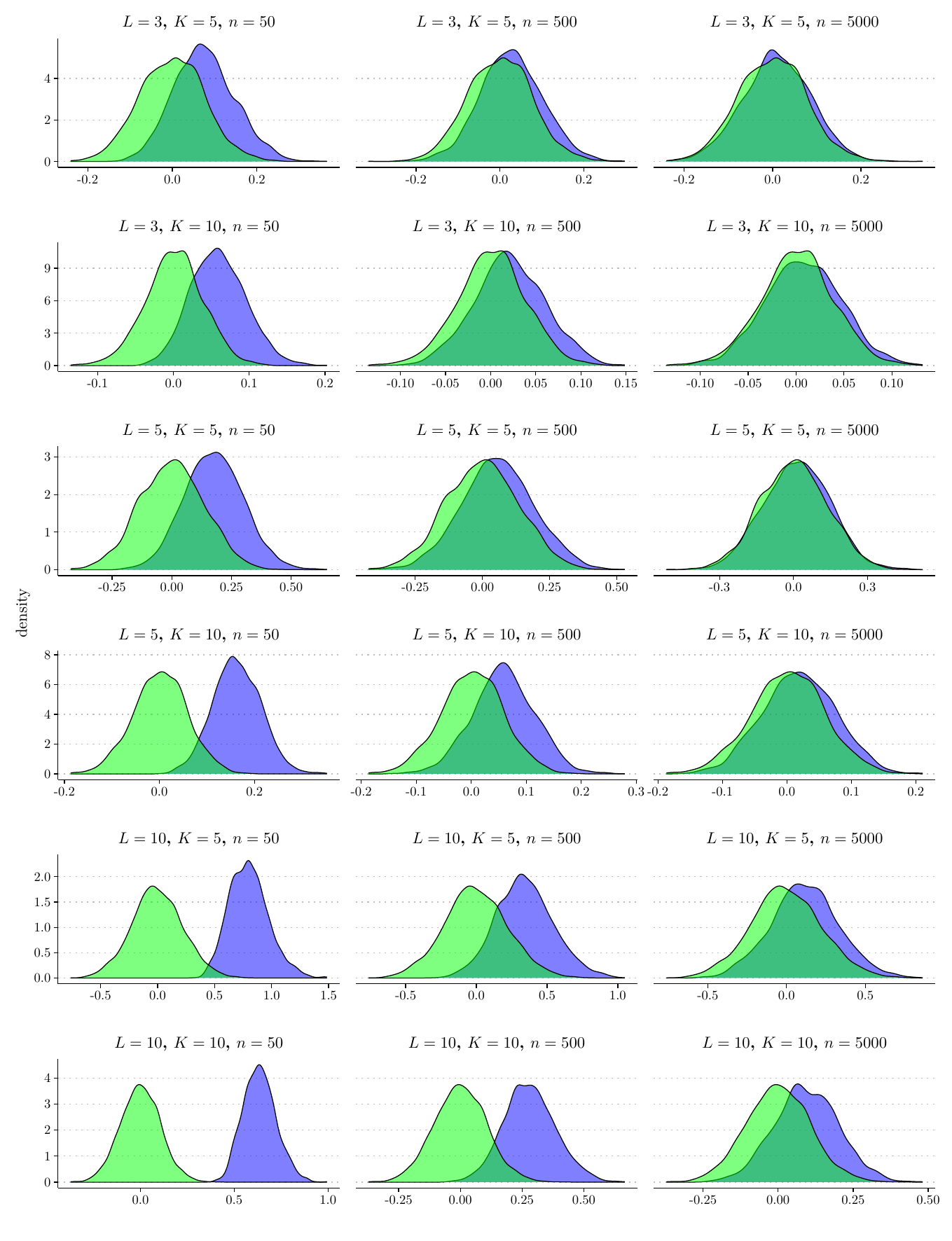}
   \caption{Comparison of the finite sample (blue) and the limiting density (green) in the one-way layout under the alternative hypothesis for different parameters.}
   \label{fig:conv_H1_one_way_kde}
\end{figure}

\paragraph{Two-Way Simulations.}\label{subsec:sim_test_two_way}

In the two-way layout, we analyze the convergence of the limit law for the FDOTT statistic given in \autoref{ex:two_way_1}. For the case that the null hypothesis holds, we generate measures $\mu^{ij}$, $i \in \idn{\K_1}, j \in \idn{\K_2}$, by drawing from the uniform distribution on $(0, 1)$ and normalizing. Then, we replace each $\mu^{ij}$ by $\bmu^{ij} \defeq \mu^{i\bullet} + \mu^{\bullet j} - \mu^{\bullet\bullet}$. If the null hypothesis does not already hold for these, then we replace those measures with negative entries by newly generated measures. Then, we again replace each measure $\mu^{ij}$ by $\bmu^{ij}$. We repeat doing this procedure until the null hypothesis holds.

As in the one-way layout, the convergence is quite fast if the null hypothesis holds, see \autoref{fig:conv_H0_two_way_kde}. Under the alternative, \autoref{fig:conv_H1_two_way_kde} again shows that convergence is significantly slower.

\begin{figure}
   \centering
   \includegraphics[width=0.95\textwidth]{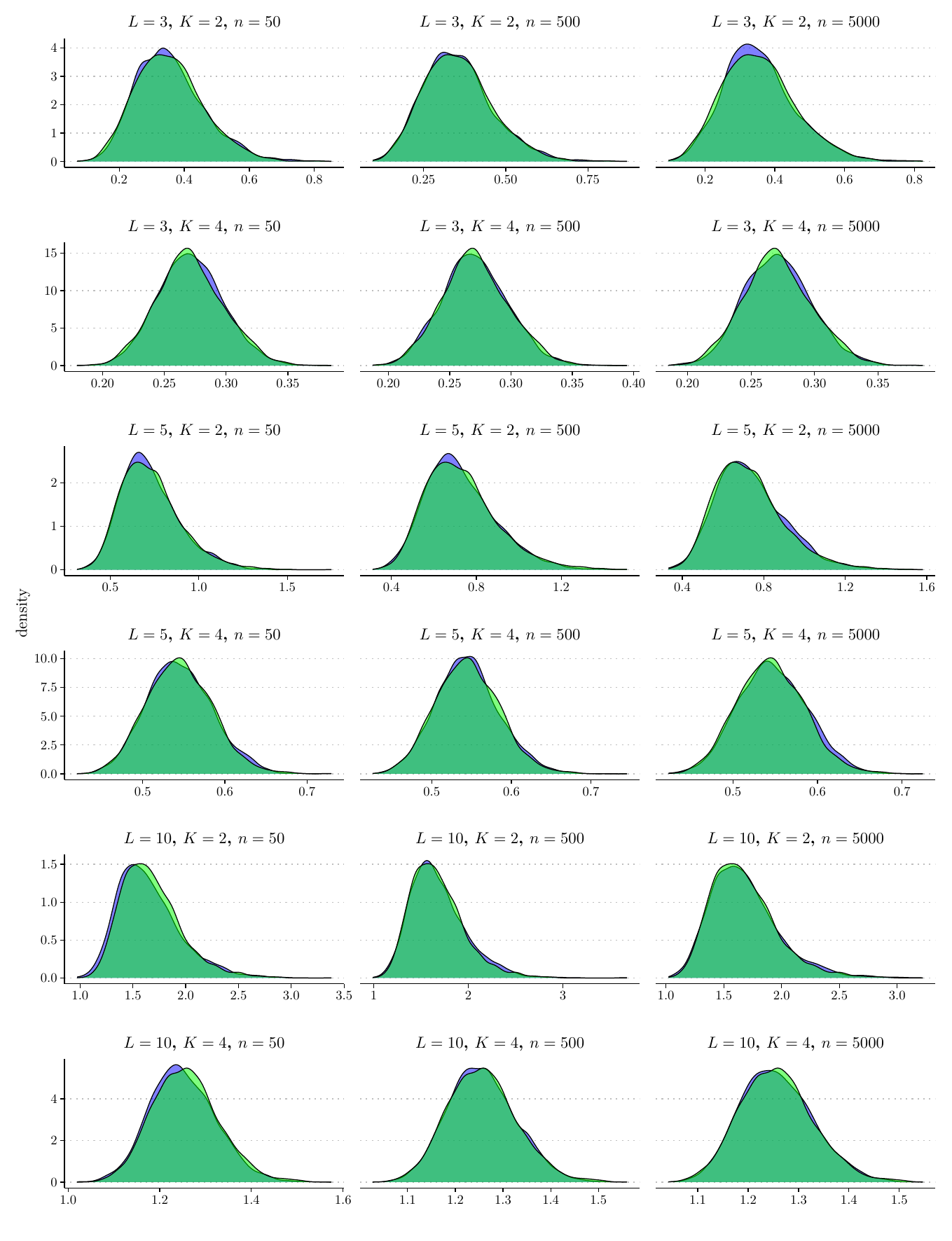}
   \caption{Comparison of the finite sample (blue) and the limiting density (green) in the two-way layout under the null hypothesis for different parameters.}
   \label{fig:conv_H0_two_way_kde}
\end{figure}

\begin{figure}
  \centering
  \includegraphics[width=0.95\textwidth]{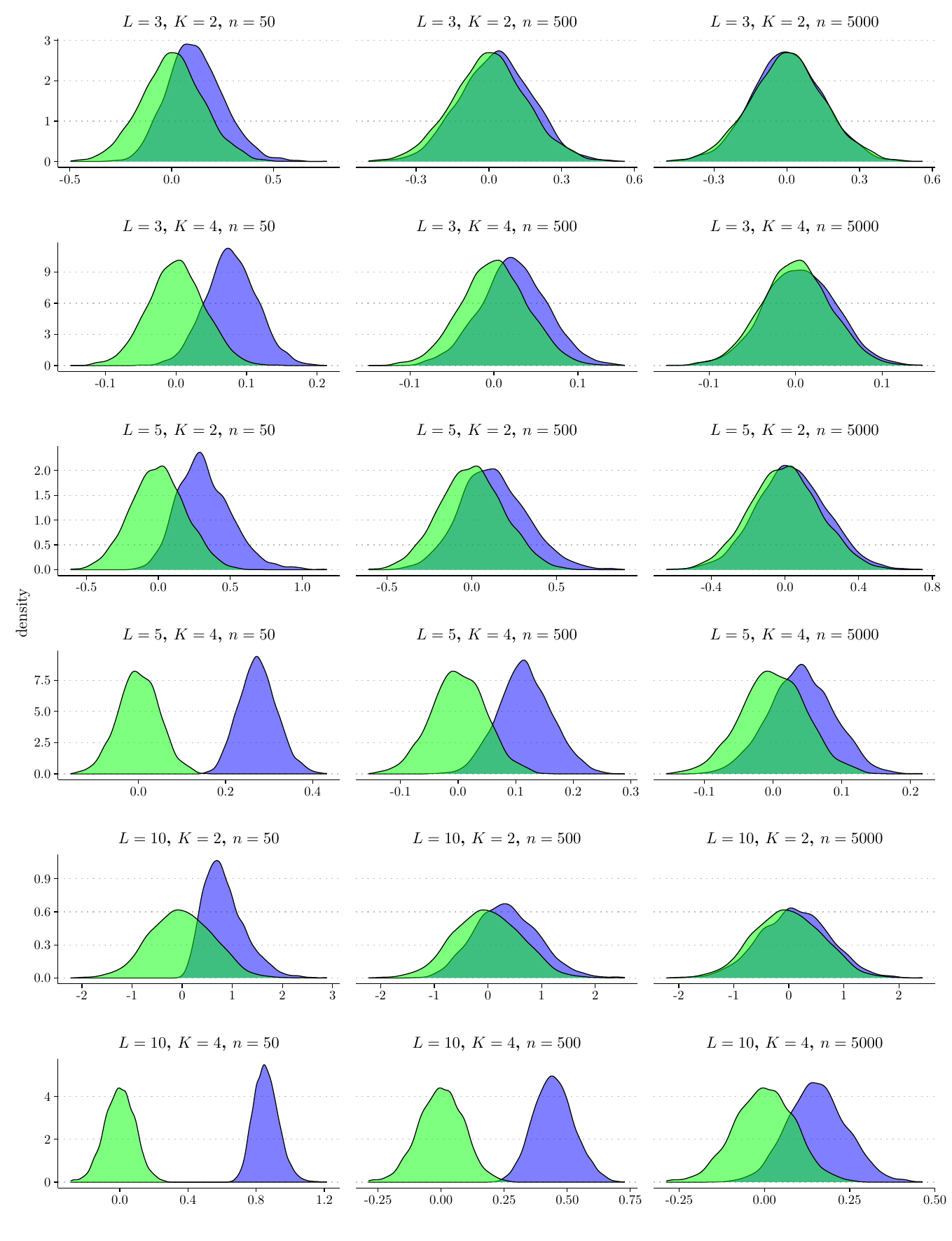}
  \caption{Comparison of the finite sample (blue) and the limiting density (green) in the two-way layout under the alternative hypothesis for different parameters.}
  \label{fig:conv_H1_two_way_kde}
\end{figure}

\paragraph{Barycenter Simulations.}

Similar to the FDOTT statistic, the distributions for the barycenter statistic converge very fast under the null hypothesis. The results can be seen in \autoref{fig:conv_H0_bary_kde}.

When the null hypothesis does not hold, then \autoref{fig:conv_H1_bary_kde} shows that the convergence is slower. Again, a slowly shrinking negative bias can be seen.

\begin{figure}
   \centering
   \includegraphics[width=0.95\textwidth]{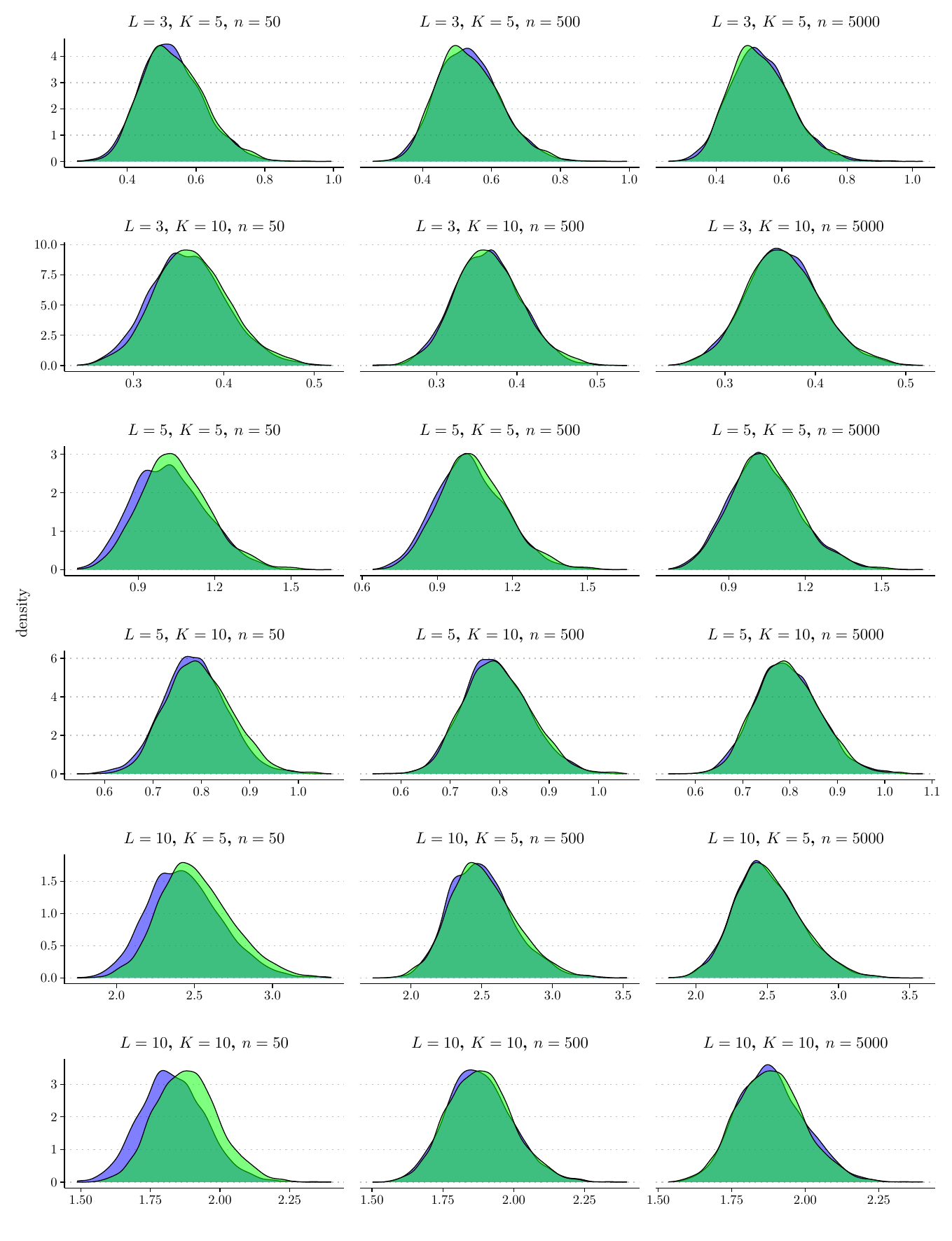}
   \caption{Comparison of the finite sample (blue) and the limiting density (green) for the barycenter method under the null hypothesis for different parameters.}
   \label{fig:conv_H0_bary_kde}
\end{figure}

\begin{figure}
   \centering
   \includegraphics[width=0.95\textwidth]{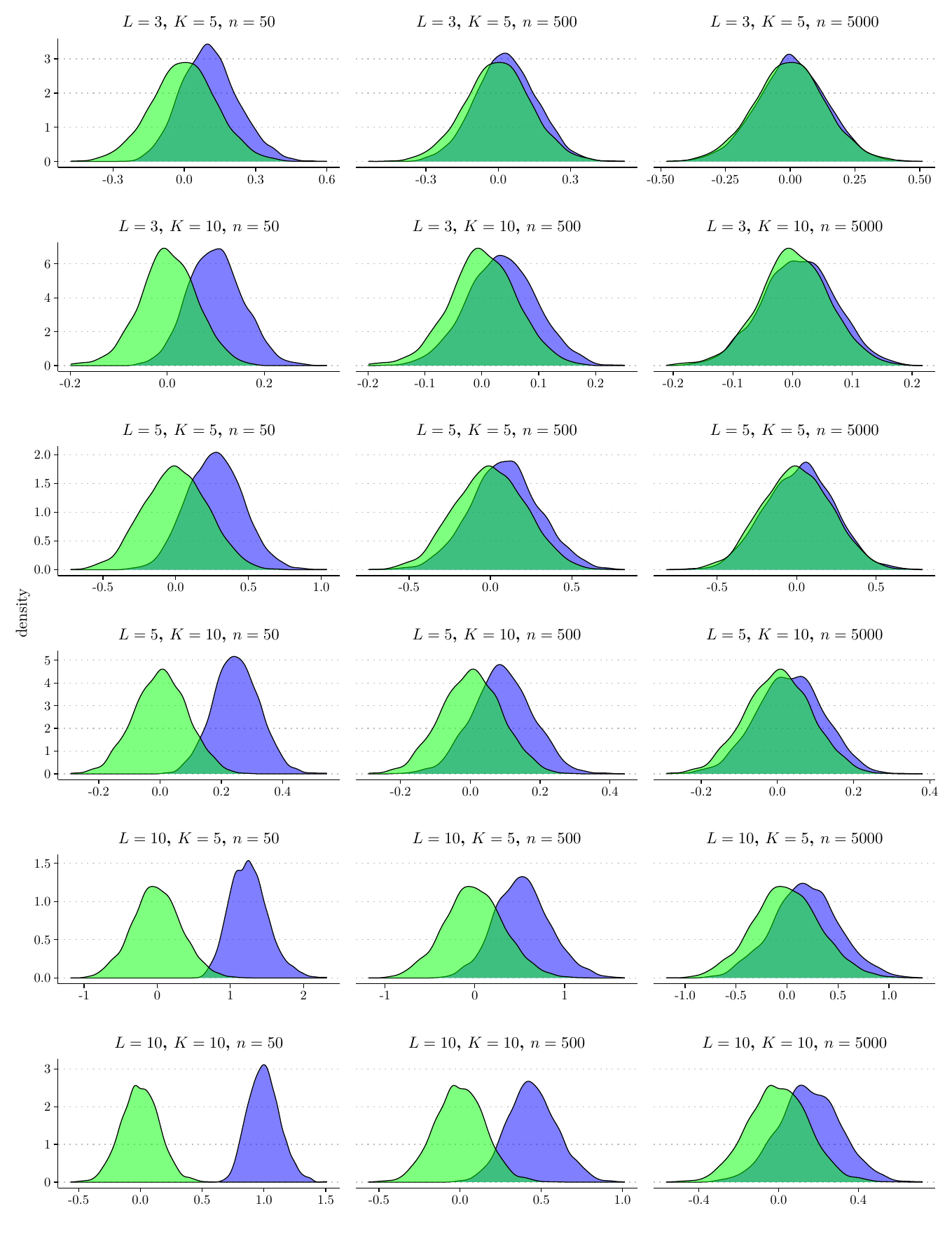}
   \caption{Comparison of the finite sample (blue) and the limiting density (green) for the barycenter method under the alternative hypothesis for different parameters.}
   \label{fig:conv_H1_bary_kde}
\end{figure}

\newpage

\subsection{Hypothesis Test Simulations} \label{subsec:sim_test_one_way}

In this subsection, we compare the performance of FDOTT in the one- and two-way layout and the barycenter method to other existing methods using synthetic data. Additionally, we compare Tukey's HSD test without (\autoref{ex:tukey_hsd_test}) or with weighting (\autoref{ex:tukey_hsd_test_w}). For this, we generate measures on the regular $L\times L$ grid and use the Euclidean distance to create the cost matrix. For each test, we sample from the measures to obtain $\hmu^1_{n_1},...,\hmu^\K_{n_{\K}}$ using one sample size $n=n_1=...=n_\K$, except for Tukey's HSD test, and then perform the particular test. For each test, we use $1000$ samples from the limiting distribution. We repeat this $250$ times and report the fraction of rejections. The level of significance is set to $\alpha = 0.05$.

\paragraph{One-way Layout.}

We compare FDOTT (one-way layout), the barycenter method, rankFD \citep{Konietschke2023}, DISCO \citep{rizzo2010disco}, PERMANOVA \citep{Anderson2001} as implemented in the \texttt{vegan} package \citep{veganPackage} and KMD \citep{Huang2024}. For FDOTT, we consider the plug-in with or without pooling (as described in \autoref{rem:enlargening}), permutation, $m$-out-of-$n$ bootstrap and derivative bootstrap approach to simulate from the limiting distribution. To apply rankFD, we map the regular grid points by enumeration to $\R$. For DISCO and PERMANOVA, we use the Euclidean distance (to the power of $1$); and employ KMD with discrete kernel and $\ell \in \{5, 10, 30\}$, the number of nearest neighbors used. 

\subparagraph{Setting I.} We choose $L=5$, $\K=6$ with sample sizes $n \in \{ 50,100, 500\}$. We consider three sets of measures $\mu^1,...,\mu^\K$. We generate $\mu^k$ from the $\mathrm{Poisson}(\lambda_k)$-distribution by normalizing its probability function on the points $0, \ldots, L^2 - 1$.
\begin{enumerate}[(i)]
  \item All measures are the same (null hypothesis): \\ 
  $\lambda = (13, 13, 13, 13, 13, 13)$; \label{item:sim_one_way_1}
  \item One measure is different (alternative hypothesis): \\
  $\lambda = (14.84, 13, 13, 13, 13, 13)$; \label{item:sim_one_way_2}
  \item All measures are different (alternative hypothesis): \\
  $\lambda = (12, 12.4, 12.8, 13.2, 13.6, 14)$. \label{item:sim_one_way_3}
\end{enumerate}

The results can be found in \autoref{tab:sim_test_one_way_1}, \autoref{tab:sim_test_one_way_2} and \autoref{tab:sim_test_one_way_3}. Under the alternatives, we see that in this case rankFD has significantly higher power, while the other methods seem to perform equally well except for FDOTT with $n^{0.8}$-out-of-$n$ bootstrap and KMD for some values of $\ell$. This highlights the influence of choosing a proper $m = o(n)$ for the $m$-out-of-$n$ bootstrap approach and $\ell$ for KMD. Furthermore, note that FDOTT with plug-in and pooling has (within the FDOTT methods) the lowest type I error under the null hypothesis, but seems to lose power under the alternative.

\begin{figure}
    \includegraphics[width=\textwidth]{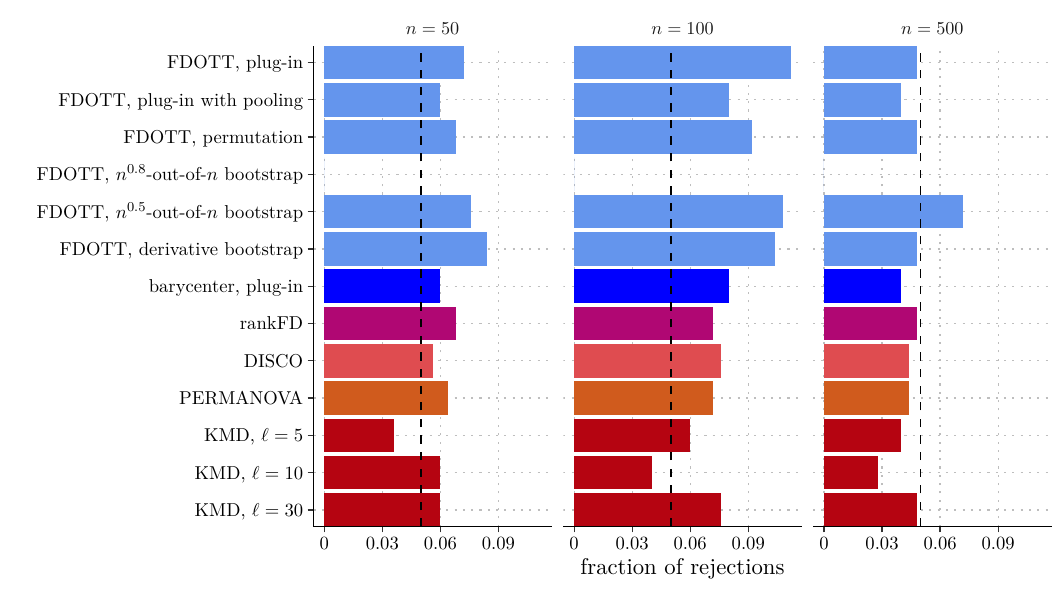}
    \caption{Fraction of rejections for the simulations in the one-way layout in setting \ref{item:sim_one_way_1}. The dashed line indicates the level of significance $\alpha=0.05$.}
    \label{tab:sim_test_one_way_1}
\end{figure}

\begin{figure}
  \includegraphics[width=\textwidth]{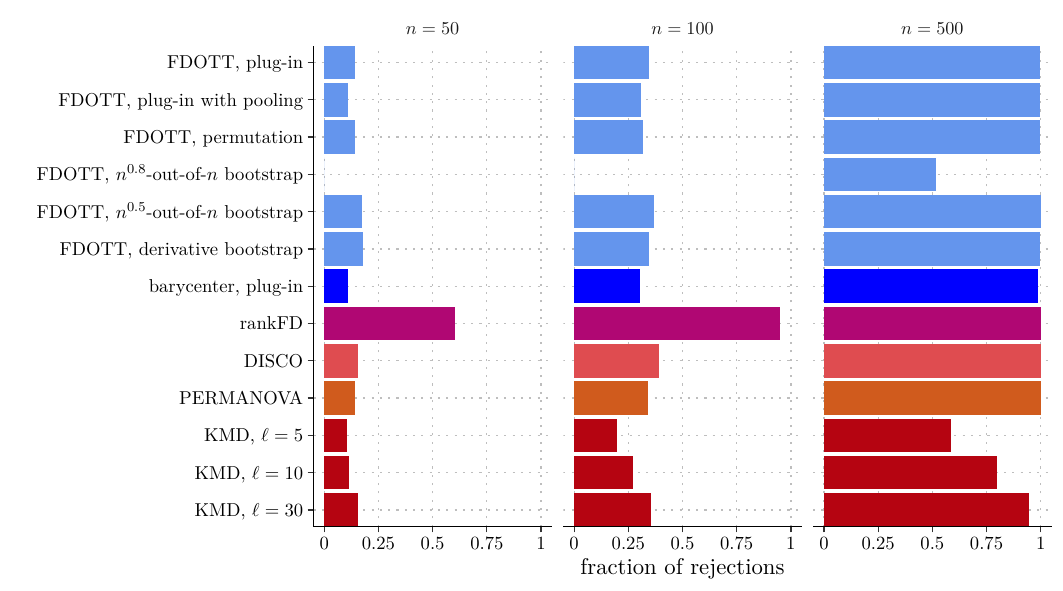}
  \caption{Fraction of rejections for the simulations in the one-way layout in setting \ref{item:sim_one_way_2}.}
  \label{tab:sim_test_one_way_2}
\end{figure}

\begin{figure}
  \includegraphics[width=\textwidth]{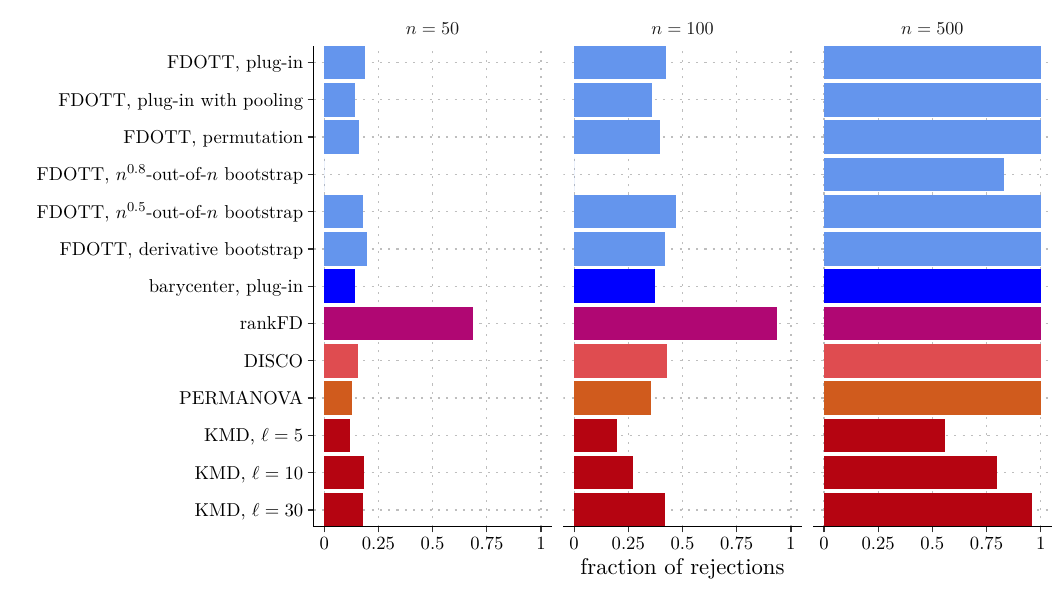}
  \caption{Fraction of rejections for the simulations in the one-way layout in setting \ref{item:sim_one_way_3}.}
  \label{tab:sim_test_one_way_3}
\end{figure}

\subparagraph{Setting II.} We consider measures that have the same mean but different covariances. We take $L = 5$, $\K=4$ with sample sizes $n \in \{ 50,100, 500\}$. Define a discretized Gaussian measure $\nu$ on $\idn{L} \times \idn{L}$ with mean $(3, 3)$ componentwise (up to normalization) as $\nu_{i,j} = \Prob{Z_1 \in [i - 0.5, i + 0.5], Z_2 \in [j - 0.5, j + 0.5]}$ for $i, j \in \idn{5}$ where $Z_1,Z_2 \sim \normaldist(3,\, 1.5^2)$ are independent. 

We use the following procedure to perturb $\nu$ in such a way that the mean stays the same: Start with $\tau \defeq \nu$, then iteratively move $\eps > 0$ mass $E$ times symmetrically around $(3,3)$. In each of the $E$ steps, we draw $I \sim \mathrm{Unif}\{1, 2, 3\}$, $J \sim \mathrm{Unif}\{1, 2\}$ and $S \sim \mathrm{Unif}\{\pm 1\}$. Let $(I_i, J_i)$, $i \in \idn{4}$, be the four symmetry points around $(3, 3)$ that include $(I, J)$, in lexicographical order. If $\tau_{I_i, J_i} \geq \eps$ for all $i \in \idn{4}$, then take $\tau_{I_i, J_i} \defeq \tau_{I_i, J_i} - S \eps$ for $i \in \{1, 4\}$ and $\tau_{I_i, J_i} \defeq \tau_{I_i, J_i} + S \eps$ for $i \in \{2, 3\}$. If $\tau_{I_i, J_i} < \eps$ for any $i \in \idn{4}$, then we do not change $\tau$ and move to the next iteration.

\begin{enumerate}[(i)] \setcounter{enumi}{3}
  \item All measures are the same (null hypothesis): \\
  $\nu = \mu^1 = \ldots = \mu^{\K}$; \label{item:sim_one_way_1_disp}
  \item One measure is different (alternative hypothesis): \\
  $\nu = \mu^1 = \ldots = \mu^{\K - 1}$ and $\mu^{\K}$ is obtained from the above procedure with $\eps=0.001$ and $E=250$; \label{item:sim_one_way_2_disp}
  \item All measures are different (alternative hypothesis): \\
  $\mu^1,\ldots,\mu^{\K}$ are all obtained by the above procedure with $\eps = 0.005$ and $E=20$. \label{item:sim_one_way_3_disp}
\end{enumerate}

The results can be found in \autoref{tab:sim_test_one_way_1_disp}, \autoref{tab:sim_test_one_way_2_disp} and \autoref{tab:sim_test_one_way_3_disp}. In this setting, we observe that compared to FDOTT and the barycenter method, the methods rankdFD, DISCO and PERMANOVA have quite low power. Only KMD can compete, which, however, highly depends on the choice of $\ell$.

\begin{figure}
    \includegraphics[width=\textwidth]{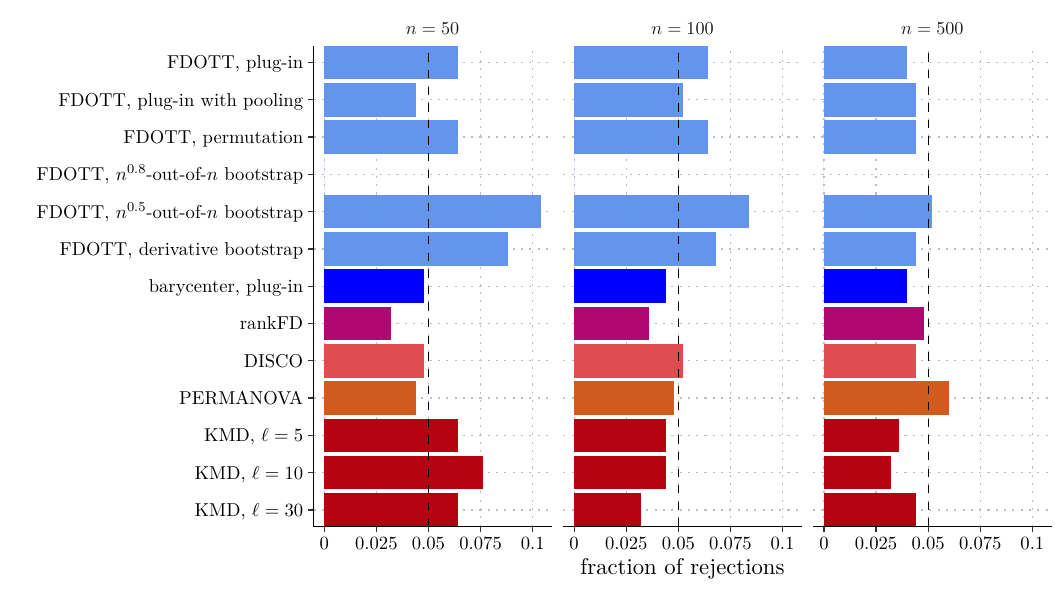}
    \caption{Fraction of rejections for the simulations in the one-way layout in setting \ref{item:sim_one_way_1_disp}. The dashed line indicates the level of significance $\alpha=0.05$.}
    \label{tab:sim_test_one_way_1_disp}
\end{figure}

\begin{figure}
  \includegraphics[width=\textwidth]{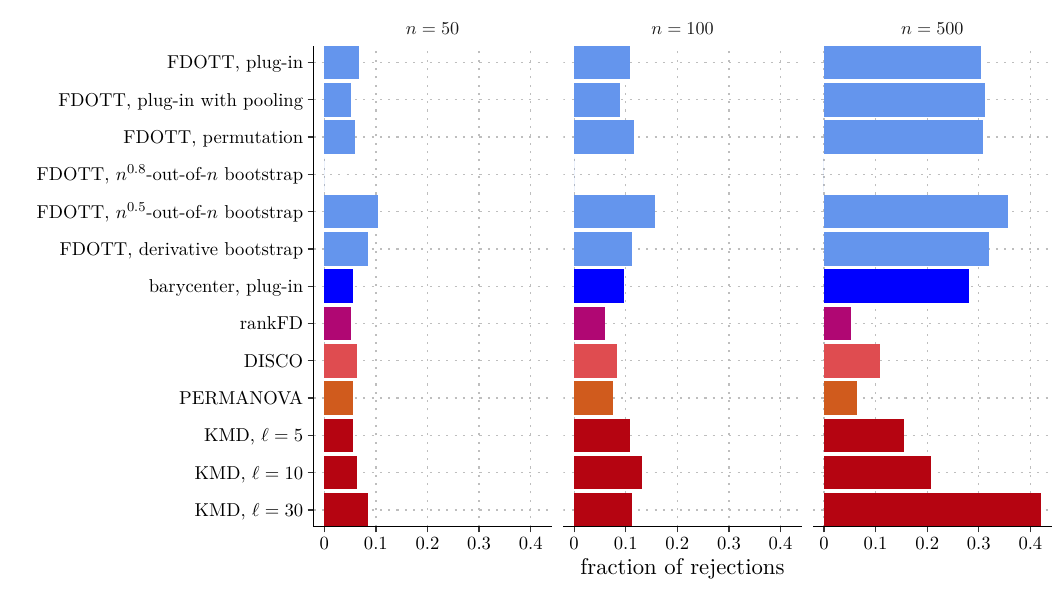}
  \caption{Fraction of rejections for the simulations in the one-way layout in setting \ref{item:sim_one_way_2_disp}.}
  \label{tab:sim_test_one_way_2_disp}
\end{figure}

\begin{figure}
  \includegraphics[width=\textwidth]{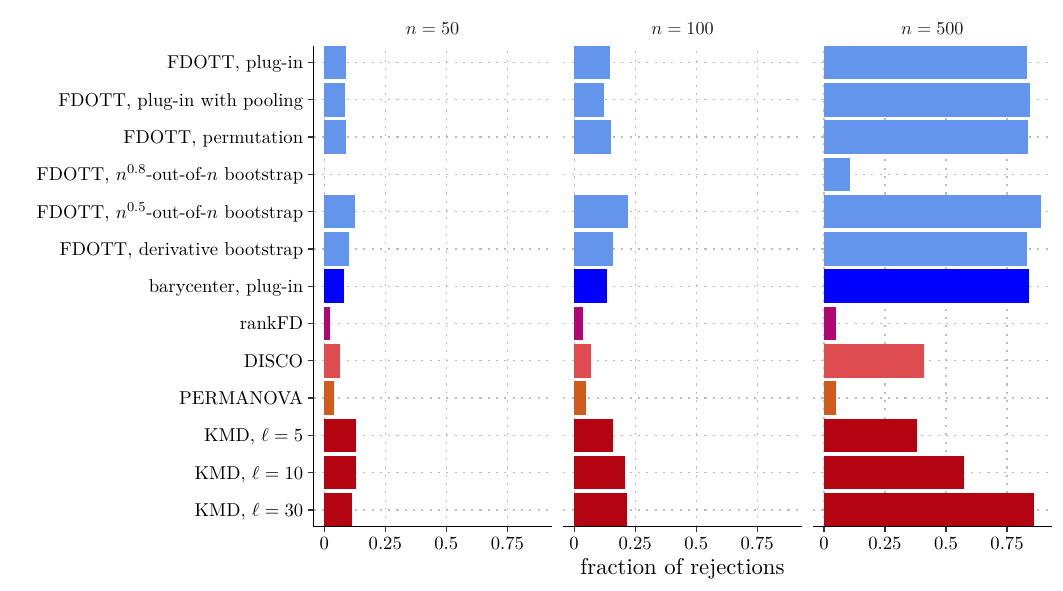}
  \caption{Fraction of rejections for the simulations in the one-way layout in setting \ref{item:sim_one_way_3_disp}.}
  \label{tab:sim_test_one_way_3_disp}
\end{figure}

\paragraph{Two-way Layout.}

For the two-way layout, we test how well FDOTT can detect an interaction effect (defined in \eqref{eq:interactionAB}) and choose $L = 5$ and $\K_1=2$, $\K_2=3$ with sample sizes $n \in \{ 50,100\}$. We compare FDOTT with plug-in, $m$-out-of-$n$ bootstrap and derivative bootstrap approach to simulate from the limiting distribution; as well as rankFD, DISCO and PERMANOVA. For rankFD, we again map the regular grid to $\R$; and for DISCO and PERMANOVA, we use the Euclidean distance (to the power of $1$). In the simulations, we use the following three sets of measures:
\begin{enumerate}[(i)]
  \item under the null hypothesis: \\
  Generated exactly as in \autoref{subsec:sim_test_two_way}; \label{item:sim_two_way_1}
  \item under the alternative, one measure is different: \\
  Generated as under the null hypothesis, but one measure is substituted by one uniformly drawn from $\Delta_{L^2}$; \label{item:sim_two_way_2}
  \item under the alternative, all measures are different: \\
   Generate the measures independently and uniformly from $\Delta_{L^2}$. \label{item:sim_two_way_3}
\end{enumerate}

The results can be seen in \autoref{tab:sim_test_two_way_1}, \autoref{tab:sim_test_two_way_2} and \autoref{tab:sim_test_two_way_3}. We see that FDOTT can detect interaction effects, even for small sample sizes. Furthermore, it appears that setting \ref{item:sim_two_way_2} is more difficult to detect than setting \ref{item:sim_two_way_3}. Moreover, the tests based on the $m$-out-of-$n$ bootstrap perform quite poorly for small $n$, in particular the choice $m=n^{0.8}$. \autoref{tab:sim_test_two_way_1} shows that both DISCO and PERMANOVA are incompatible with our definition of interaction. In particular, their notion of interaction is not implied by ours. For rankFD this is not the case. However, especially in setting \ref{item:sim_two_way_2} rankFD seems to struggle to reject the interaction.

\begin{figure}
  \includegraphics[width=\textwidth]{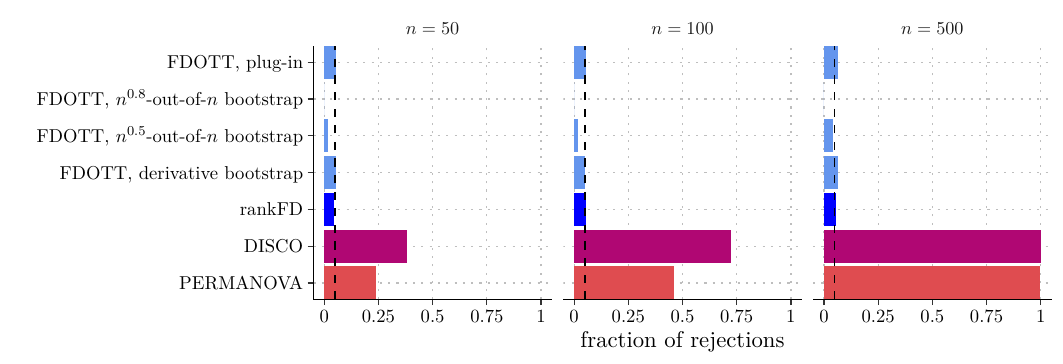}
  \caption{Fraction of rejections for the simulations in the two-way layout in setting \ref{item:sim_two_way_1}. The dashed line indicates the level of significance $\alpha=0.05$.}
  \label{tab:sim_test_two_way_1}
\end{figure}
\begin{figure}
  \includegraphics[width=\textwidth]{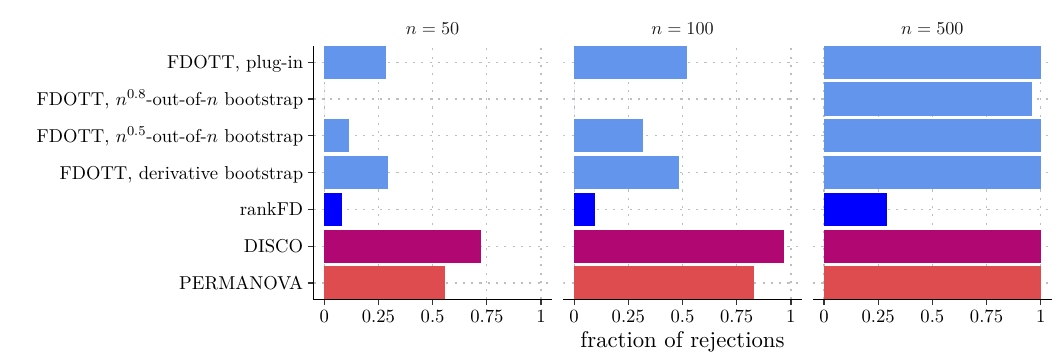}
  \caption{Fraction of rejections for the simulations in the two-way layout in setting \ref{item:sim_two_way_2}.}
  \label{tab:sim_test_two_way_2}
\end{figure}

\begin{figure}
  \includegraphics[width=\textwidth]{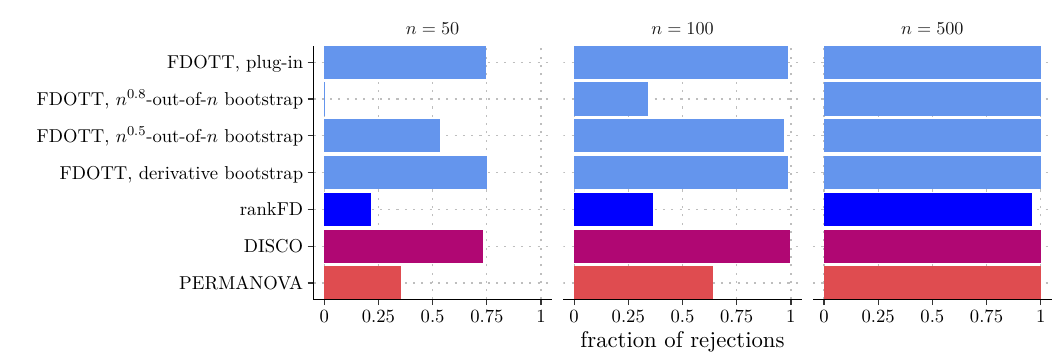}
  \caption{Fraction of rejections for the simulations in the two-way layout in setting \ref{item:sim_two_way_3}.}
  \label{tab:sim_test_two_way_3}
\end{figure}

\paragraph{Tukey's HSD Test.}

For the simulations regarding Tukey's HSD test, we chose $L=5$, $\K=4$ with unequal sample sizes $(n_1, n_2, n_3, n_4) \in \{ (20, 50, 70, 30),\allowbreak (60, 150, 210, 90), (160, 400, 560, 240)\}$. We consider four sets of measures $\mu^1,...,\mu^\K$ that are again generated from the $\mathrm{Poisson}(\lambda_k)$-distribution:
\begin{enumerate}[(i)]
  \item All measures are the same (null hypothesis): \\
   $\lambda = (13, 13, 13, 13)$; \label{item:sim_hsd_1}
  \item One measure is different (alternative hypothesis): \\
   $\lambda = (16, 13, 13, 13)$; \label{item:sim_hsd_2}
  \item One measure is different (alternative hypothesis): \\
  $\lambda = (13, 13, 16, 13)$; \label{item:sim_hsd_3}
  \item All measures are different (alternative hypothesis): \\
  $\lambda = (11, 12 + 1/3, 13 + 2/3, 15)$. \label{item:sim_hsd_4}
\end{enumerate}
In each setting, we compare Tukey's HSD test without or with weighting. The results are given in \autoref{tab:sim_hsd_1}, \autoref{tab:sim_hsd_2}, \autoref{tab:sim_hsd_3} and \autoref{tab:sim_hsd_4}. We see that under the null hypothesis, both tests, i.e., with and without weighting, perform very similar and both keep the level $\alpha = 0.05$. Under the alternative in setting \ref{item:sim_hsd_2}, the test without weighting seems to perform better for $n = (60,150,210,90)$ (up to $0.24$) than the test with weighting, while in setting \ref{item:sim_hsd_3} it is the other way round (up to $0.49$). Lastly, in setting \ref{item:sim_hsd_4} Tukey's HSD test with weighting again performs significantly better (up to $0.63$), except for $n = (20, 50, 70, 30)$.

\begin{table}
  \centering
\begin{tabular}{c|l|rrrrrr}
 sample sizes & weighting & $\mu^1 = \mu^2$ & $\mu^1 = \mu^3$ & $\mu^1 = \mu^4$ & $\mu^2 = \mu^3$ & $\mu^2 = \mu^4$ & $\mu^3 = \mu^4$ \\\hline
 (20, 50, &no  & 0.020 & 0.016 & 0.036 & 0.000 & 0.012 & 0.008\\
70, 30) &yes   & 0.004 & 0.016 & 0.016 & 0.028 & 0.036 & 0.012\\\hline
(60, 150, &no  & 0.036 & 0.016 & 0.040 & 0.000 & 0.004 & 0.000\\
210, 90) &yes  & 0.016 & 0.012 & 0.012 & 0.000 & 0.008 & 0.008\\\hline
(160, 400 &no  & 0.004 & 0.000 & 0.036 & 0.000 & 0.000 & 0.000\\
560, 240) &yes & 0.004 & 0.004 & 0.008 & 0.012 & 0.012 & 0.004\\
\end{tabular}
  \caption{Fraction of rejections for the simulations of Tukey's HSD test without or with weighting in setting \ref{item:sim_hsd_1}.} \label{tab:sim_hsd_1}
\end{table}

\begin{table}
  \centering
\begin{tabular}{c|l|rrrrrr}
 sample sizes & weighting & $\mu^1 = \mu^2$ & $\mu^1 = \mu^3$ & $\mu^1 = \mu^4$ & $\mu^2 = \mu^3$ & $\mu^2 = \mu^4$ & $\mu^3 = \mu^4$ \\\hline
 (20, 50, &no  & 0.104 & 0.120 & 0.140 & 0.000 & 0.012 & 0.000\\
70, 30) &yes   & 0.064 & 0.100 & 0.052 & 0.024 & 0.040 & 0.012\\\hline
(60, 150, &no  & 0.624 & 0.596 & 0.656 & 0.000 & 0.004 & 0.000\\
210, 90) &yes  & 0.536 & 0.608 & 0.420 & 0.000 & 0.008 & 0.008\\\hline
(160, 400 &no  & 1.000 & 1.000 & 1.000 & 0.000 & 0.000 & 0.000\\
560, 240) &yes & 1.000 & 1.000 & 0.988 & 0.012 & 0.012 & 0.004\\
\end{tabular}
  \caption{Fraction of rejections for the simulations of Tukey's HSD test without or with weighting in setting \ref{item:sim_hsd_2}.} \label{tab:sim_hsd_2}
\end{table}

\begin{table}
  \centering
  \begin{tabular}{c|l|rrrrrr}
 sample sizes & weighting & $\mu^1 = \mu^2$ & $\mu^1 = \mu^3$ & $\mu^1 = \mu^4$ & $\mu^2 = \mu^3$ & $\mu^2 = \mu^4$ & $\mu^3 = \mu^4$ \\\hline
 (20, 50, &no  & 0.020 & 0.092 & 0.036 & 0.044 & 0.012 & 0.056\\
70, 30) &yes   & 0.008 & 0.076 & 0.016 & 0.312 & 0.036 & 0.132\\\hline
(60, 150, &no  & 0.036 & 0.664 & 0.040 & 0.496 & 0.004 & 0.628\\
210, 90) &yes  & 0.020 & 0.648 & 0.008 & 0.984 & 0.004 & 0.884\\\hline
(160, 400 &no  & 0.004 & 1.000 & 0.036 & 1.000 & 0.000 & 1.000\\
560, 240) &yes & 0.004 & 1.000 & 0.008 & 1.000 & 0.012 & 1.000\\
\end{tabular}
  \caption{Fraction of rejections for the simulations of Tukey's HSD test without or with weighting in setting \ref{item:sim_hsd_3}.} \label{tab:sim_hsd_3}
\end{table}

\begin{table}
  \centering
\begin{tabular}{c|l|rrrrrr}
 sample sizes & weighting & $\mu^1 = \mu^2$ & $\mu^1 = \mu^3$ & $\mu^1 = \mu^4$ & $\mu^2 = \mu^3$ & $\mu^2 = \mu^4$ & $\mu^3 = \mu^4$ \\\hline
 (20, 50, &no  & 0.028 & 0.056 & 0.408 & 0.000 & 0.052 & 0.012\\
70, 30) &yes   & 0.012 & 0.056 & 0.172 & 0.032 & 0.092 & 0.032\\\hline
(60, 150, &no  & 0.032 & 0.456 & 0.992 & 0.000 & 0.420 & 0.016\\
210, 90) &yes  & 0.020 & 0.464 & 0.928 & 0.104 & 0.656 & 0.064\\\hline
(160, 400 &no  & 0.292 & 1.000 & 1.000 & 0.044 & 0.996 & 0.096\\
560, 240) &yes & 0.224 & 1.000 & 1.000 & 0.676 & 1.000 & 0.328\\
\end{tabular}
  \caption{Fraction of rejections for the simulations of Tukey's HSD test without or with weighting in setting \ref{item:sim_hsd_4}.} \label{tab:sim_hsd_4}
\end{table}

\subsection{Local Power}

In this subsection, we perform simulations in the one-way layout to see which test performs better under local alternatives. To this end, recall \autoref{subsec:local_power}. Using the limiting distributions derived there, we can simulate the probabilities of rejections for given $\mu$ and $\nu$. The level of significance is set to $\alpha = 0.05$. We generate the measures on the regular $L \times L$ grid with $L = 5$, $\K = 6$ and with equal sample sizes $n = n_1 = \ldots = n_{\K}$. Again, we use the $\mathrm{Poisson}(\lambda_k)$-distribution: $\mu^1 = \ldots = \mu^{\K}$ are calculated with $\lambda = (13, 13, 13, 13, 13, 13)$ and for $\nu^1, \ldots, \nu^{\K}$ we consider the following cases for $\tilde{\lambda}$:
\begin{enumerate}[(i)]
  \item $\tilde{\lambda} = (15, 13, 13, 13, 13, 13)$; \label{item:sim_local_power_1}
  \item $\tilde{\lambda} = (12.0, 12.4, 12.8, 13.2, 13.6, 14.0)$; \label{item:sim_local_power_2}
  \item $\tilde{\lambda} = (1, 13, 13, 13, 13, 13)$; \label{item:sim_local_power_3}
  \item $\tilde{\lambda} = (5, 9, 13, 17, 21, 25)$; \label{item:sim_local_power_4}
  \item $\tilde{\lambda} = (0, 13, 13, 13, 13, 13)$; \label{item:sim_local_power_5}
  \item $\tilde{\lambda} = (0, 6, 12, 18, 24, 30)$. \label{item:sim_local_power_6}
\end{enumerate}

The results can be seen in \autoref{tab:sim_local_power}. We see that for most settings there are only small differences in local power between FDOTT and the barycenter method. Notably, in setting \ref{item:sim_local_power_5} FDOTT's local power is greater by $0.1$.

\begin{table}
\centering
\begin{tabular}{l|r|r}
setting & FDOTT & barycenter method \\\hline
\ref{item:sim_local_power_1} & 0.059 & 0.057\\
\ref{item:sim_local_power_2} & 0.061 & 0.057\\
\ref{item:sim_local_power_3} & 0.273 & 0.223\\
\ref{item:sim_local_power_4} & 0.298 & 0.284\\
\ref{item:sim_local_power_5} & 0.447 & 0.348\\
\ref{item:sim_local_power_6} & 0.860 & 0.790\\
\end{tabular}
\caption{Simulated local power for FDOTT and the barycenter method in the one-way layout for different settings.} \label{tab:sim_local_power}
\end{table}

\end{document}